\pdfoutput=1
\documentclass[hidelinks, reqno]{amsart}
\usepackage{graphicx}
\usepackage{subcaption}
\usepackage{tabularx}
\usepackage{booktabs}
\usepackage{array}
\usepackage{amsmath}
\usepackage{amsfonts}
\usepackage{amssymb}
\usepackage{amsthm}
\usepackage[scr]{rsfso}
\usepackage{enumitem}
\usepackage{permute}
\usepackage{slashed}
\usepackage[usenames,dvipsnames]{xcolor}
\usepackage[pagebackref = true, colorlinks, linkcolor = Red, citecolor = Green, bookmarksdepth=2, linktocpage=true]{hyperref}
\usepackage[capitalize]{cleveref}
\usepackage{comment}
\usepackage{changepage}

\usepackage[T1]{fontenc}
\usepackage{makecell}

\usepackage{tikz}
\usetikzlibrary{calc}

\allowdisplaybreaks

\DeclareMathOperator{\interior}{int}

\DeclareMathOperator{\Id}{Id}

\DeclareMathOperator{\Stab}{Stab}

\DeclareMathOperator{\supp}{supp}

\DeclareMathOperator{\SO}{SO}
\DeclareMathOperator{\SU}{SU}
\DeclareMathOperator{\Sp}{Sp}
\DeclareMathOperator{\U}{U}
\DeclareMathOperator{\F}{F}
\DeclareMathOperator{\Special}{S}
\DeclareMathOperator{\Spin}{Spin}

\DeclareMathOperator{\Lip}{Lip}

\DeclareMathOperator{\proj}{proj}

\DeclareMathOperator{\T}{T}
\DeclareMathOperator{\ExceptionalLieF}{F}
\DeclareMathOperator{\Hol}{H}

\DeclareMathOperator{\Isom}{Isom}

\DeclareMathOperator{\diam}{diam}

\DeclareMathOperator{\len}{len}
\let\Pr\relax
\DeclareMathOperator{\Pr}{Pr}

\DeclareMathOperator{\ad}{ad}
\DeclareMathOperator{\Ad}{Ad}
\DeclareMathOperator{\BP}{BP}

\DeclareMathOperator{\Core}{Core}
\DeclareMathOperator{\Hull}{Hull}

\DeclareMathOperator{\im}{im}

\newcommand{\N}{\mathbb{N}}

\newcommand{\R}{\mathbb{R}}
\newcommand{\C}{\mathbb{C}}

\newcommand{\LieG}{\mathfrak{g}}
\newcommand{\LieA}{\mathfrak{a}}
\newcommand{\LieN}{\mathfrak{n}}
\newcommand{\LieK}{\mathfrak{k}}
\newcommand{\LieM}{\mathfrak{m}}
\newcommand{\LieP}{\mathfrak{p}}

\newcommand{\LieSimple}{\mathfrak{s}}

\newcommand{\limitset}{\Lambda_\Gamma}
\newcommand{\loglimitset}{\Lambda_\Gamma^{\log}}

\DeclareFontFamily{U}{mathb}{\hyphenchar\font45}
\DeclareFontShape{U}{mathb}{m}{n}{
<5> <6> <7> <8> <9> <10> gen * mathb
<10.95> mathb10 <12> <14.4> <17.28> <20.74> <24.88> mathb12
}{}
\DeclareSymbolFont{mathb}{U}{mathb}{m}{n}
\DeclareMathSymbol{\bigast}{2}{mathb}{"06}

\def\XXint#1#2#3{{\setbox0=\hbox{$#1{#2#3}{\int}$}
\vcenter{\hbox{$#2#3$}}\kern-.5\wd0}}

\theoremstyle{plain}
\newtheorem{theorem}{Theorem}[section]
\newtheorem{proposition}[theorem]{Proposition}
\newtheorem{lemma}[theorem]{Lemma}
\newtheorem{corollary}[theorem]{Corollary}

\theoremstyle{definition}
\newtheorem{definition}[theorem]{Definition}

\theoremstyle{remark}
\newtheorem{remark}[theorem]{Remark}

\setlist[enumerate,1]{ref=(\arabic*)}
\setlist[enumerate,2]{ref=(\theenumi)(\alph*)}
\setlist[enumerate,3]{ref=(\theenumi)(\theenumii)(\roman*)}
\setlist[enumerate,4]{ref=(\theenumi)(\theenumii)(\theenumiii)(\Alph*)}

\newlist{alternative}{enumerate}{4}     
\setlist[alternative,1]{label=(\arabic*), ref=(\arabic*)}
\setlist[alternative,2]{label=(\alph*), ref=(\thealternativei)(\alph*)}
\setlist[alternative,3]{label=(\roman*), ref=(\thealternativei)(\thealternativeii)(\roman*)}
\setlist[alternative,4]{label=(\Alph*), ref=(\thealternativei)(\thealternativeii)(\thealternativeiii)(\Alph*)}

\Crefname{enumi}{Property}{Properties}
\Crefname{alternativei}{Alternative}{Alternatives}
\Crefname{section}{\S}{\S\S}
\Crefname{subsection}{\S}{\S\S}

\begin{document}


\title[Exponential mixing of frame flows]{Exponential mixing of frame flows for convex cocompact locally symmetric spaces}

\author{Michael Chow}
\address{Department of Mathematics, Yale University, New Haven, Connecticut 06511}
\email{mikey.chow@yale.edu}

\author{Pratyush Sarkar}
\address{Department of Mathematics, UC San Diego, San Diego, California 92093}
\email{psarkar@ucsd.edu}

\date{\today}

\begin{abstract}
Let $G$ be a connected center-free simple real algebraic group of rank one and $\Gamma < G$ be a Zariski dense torsion-free convex cocompact subgroup. We prove that the frame flow on $\Gamma \backslash G$, i.e., the right translation action of a one-parameter subgroup $\{a_t\}_{t \in \R} < G$ of semisimple elements, is exponentially mixing with respect to the Bowen--Margulis--Sullivan measure. The key step is proving suitable generalizations of the local non-integrability condition and the non-concentration property which are essential for Dolgopyat's method. This generalizes the work of Sarkar--Winter for $G = \SO(n, 1)^\circ$ and also strengthens the mixing result of Winter in the convex cocompact case.
\end{abstract}

\maketitle


\setcounter{tocdepth}{1}
\tableofcontents

\section{Introduction}
Let $\mathbb{H}_\mathbb{K}^n$ for $\mathbb{K} \in \{\mathbb{R}, \mathbb{C}, \mathbb{H}, \mathbb{O}\}$ be the real, complex, quaternionic, or octonionic hyperbolic spaces indexed by integers $n \geq 2$ ($n = 2$ if $\mathbb{K} = \mathbb{O}$). Let $G = \Isom^+(\mathbb{H}_\mathbb{K}^n)$ be its group of orientation preserving isometries. Let $\Gamma < G$ be a torsion-free discrete subgroup and $X = \Gamma \backslash \mathbb{H}_\mathbb{K}^n$. Any (irreducible) rank one locally symmetric space (of noncompact type) is obtained in this fashion. We can identify $X$, its unit tangent bundle $\T^1(X)$, and its holonomy bundle $\Hol(X)$ with $\Gamma \backslash G/K$, $\Gamma \backslash G/M$, and $\Gamma \backslash G$, where $M < K$ are appropriate compact subgroups of $G$. The holonomy bundle $\Hol(X)$ can be realized as a subbundle of the frame bundle $\F(X)$ and it can also be thought of as the $\mathbb{K}$-frame bundle for $\mathbb{K} \in \{\mathbb{R}, \mathbb{C}, \mathbb{H}\}$. It is natural to consider this subbundle because it is an invariant subset of $\F(X)$ under the frame flow. Let $A = \{a_t\}_{t \in \R} < G$ be a one-parameter subgroup of semisimple elements such that its right translation action corresponds to the geodesic flow on $\Gamma \backslash G/M$ and the frame flow on $\Gamma \backslash G$. Let $\mathsf{m}$ be the Bowen--Margulis--Sullivan probability measure on $\Gamma \backslash G$, i.e., the $M$-invariant lift of the one on $\Gamma \backslash G/M$ which is known to be the measure of maximal entropy. Whether the frame flow is \emph{exponentially} mixing with respect to $\mathsf{m}$ is a fundamental question in the study of dynamics. If $\Gamma$ is a lattice, $\mathsf{m}$ coincides with the $G$-invariant probability measure and there is a rich history in the literature establishing exponential mixing of the frame flow with respect to $\mathsf{m}$ using representation theory (see for example \cite{Rat87,Moo87}). Such results have many applications including orbit counting, equidistribution, prime geodesic theorems, and shrinking target problems (see for example \cite{DRS93,EM93,BO12,MMO14,MO15,KO21,TW22}).

In this paper, our hypothesis is that $\Gamma$ is \emph{Zariski dense and convex cocompact}. The first condition is necessary because otherwise it is not hard to see that the frame flow on $\Gamma \backslash G$ is not even ergodic (cf. \cite{FS90,Win15,SW21}, although it is not mentioned explicitly). The second technical condition means that the smallest closed convex subset of $X$ containing all closed geodesics is compact. Although it does not allow the existence of cusps, it gives a large class of rank one locally symmetric spaces which includes the compact ones.

Let us assume the aforementioned setting. One can still prove exponential mixing of the frame flow using representation theory provided that there is an appropriate spectral gap as in the work of Mohammadi--Oh \cite{MO15} for a certain subclass of \emph{real} hyperbolic manifolds (see also the further work of Edwards--Oh \cite{EO21} for the \emph{geodesic} flow). In general, however, such a spectral gap does not exist. Using an alternative approach, the works of Babillot \cite{Bab02} and Winter \cite{Win15} show that the frame flow is mixing. Using Dolgopyat's method \cite{Dol98}, Stoyanov \cite{Sto11} proved exponential mixing of the \emph{geodesic} flow (see also the works of Li--Pan \cite{LP22} and Khalil \cite{Kha21} for the geometrically finite case). Recently, by developing a frame flow version of Dolgopyat's method, Sarkar--Winter \cite{SW21} proved exponential mixing of the frame flow for \emph{real} hyperbolic manifolds. In this paper, we generalize the result of Sarkar--Winter to rank one locally symmetric spaces. For any $\alpha \in (0, 1]$, we denote by $C_{\mathrm{c}}^{0,\alpha}(\Gamma \backslash G)$ the space of compactly supported $\alpha$-H\"{o}lder continuous functions on $\Gamma \backslash G$ endowed with the $\alpha$-H\"{o}lder norm $\|\cdot\|_{C^{0,\alpha}}$.

\begin{theorem}
\label{thm:TheoremExponentialMixingOfFrameFlow}
Let $\alpha \in (0, 1]$. There exist $\eta_\alpha > 0$ and $C > 0$ (independent of $\alpha$) such that for all $\phi, \psi \in C_{\mathrm{c}}^{0,\alpha}(\Gamma \backslash G)$ and $t > 0$, we have
\begin{align*}
\left|\int_{\Gamma \backslash G} \phi(xa_t)\psi(x) \, d\mathsf{m}(x) - \mathsf{m}(\phi) \cdot \mathsf{m}(\psi)\right| \leq C e^{-\eta_\alpha t} \|\phi\|_{C^{0,\alpha}} \|\psi\|_{C^{0,\alpha}}.
\end{align*}
\end{theorem}

Fix a right $G$-invariant measure on $\Gamma \backslash G$ induced by some fixed Haar measure on $G$. We denote by $m^{\mathrm{BR}}$ and $m^{\mathrm{BR}_*}$ the unstable and stable Burger--Roblin measures on $\Gamma \backslash G$ respectively, compatible with the choice of the Haar measure. Using Roblin's transverse intersection argument \cite{Rob03,OS13,OW16}, we can derive the following theorem regarding decay of matrix coefficients from \cref{thm:TheoremExponentialMixingOfFrameFlow}.

\begin{theorem}
\label{thm:TheoremDecayOfMatrixCoefficients}
Let $\alpha \in (0, 1]$. There exists $\eta_\alpha > 0$ such that for all $\phi, \psi \in C_{\mathrm{c}}^{0,\alpha}(\Gamma \backslash G)$, there exists $C > 0$ (depending only on $\supp(\phi)$ and $\supp(\psi)$, and independent of $\alpha$) such that for all $t > 0$, we have
\begin{align*}
\left|e^{(D_{\mathbb{K}, n} - \delta_\Gamma)t}\int_{\Gamma \backslash G} \phi(xa_t)\psi(x) \, dx - m^{\mathrm{BR}}(\phi) \cdot m^{\mathrm{BR}_*}(\psi)\right| \leq C e^{-\eta_\alpha t} \|\phi\|_{C^{0,\alpha}} \|\psi\|_{C^{0,\alpha}}
\end{align*}
where $D_{\mathbb{K}, n} := \dim_\R(\mathbb{K})n+\dim_\R(\mathbb{K})-2$ is the volume entropy of $\mathbb{H}_\mathbb{K}^n$. 
\end{theorem}

\subsection{Outline of the proof of \texorpdfstring{\cref{thm:TheoremExponentialMixingOfFrameFlow}}{\autoref{thm:TheoremExponentialMixingOfFrameFlow}}}
The argument in this paper is along similar lines as that of \cite{SW21}. Throughout this paper, we omit proofs which are verbatim repetitions of the ones in \cite{SW21} and focus on the difficulties that arise for general rank one locally symmetric spaces.

Firstly, convex cocompactness of $\Gamma$ implies that the nonwandering set $\supp(\mathsf{m})$ of the geodesic flow, which is Anosov, is compact and hence by the works of Bowen \cite{Bow70,Bow73}, Ratner \cite{Rat73}, and later Pollicott \cite{Pol87}, there exists a Markov section on $\supp(\mathsf{m})$. This allows one to use techniques from thermodynamic formalism. It is now well known that a given flow is exponentially mixing if one obtains appropriate spectral bounds for the associated transfer operators \cite{Pol85}. Viewing the frame flow as an extension of the geodesic flow by the compact group $M$, we are lead to consider the transfer operators with \emph{holonomy} which are twisted by irreducible representations of $M$. More precisely, for a given $\xi = a + ib \in \mathbb C$ and irreducible representation $\rho: M \to \U(V_\rho)$, we consider $\mathcal{M}_{\xi, \rho}: C\bigl(U, V_\rho^{\oplus \dim(\rho)}\bigr) \to C\bigl(U, V_\rho^{\oplus \dim(\rho)}\bigr)$ defined by
\begin{align*}
\mathcal{M}_{\xi, \rho}(H)(u) &= \sum_{u' \in \sigma^{-1}(u)} e^{-(a + \delta_\Gamma - ib)\tau(u')} \rho(\vartheta(u')^{-1}) H(u')
\end{align*}
where $\tau: U \to \mathbb R$ is the first return time map and $\vartheta: U \to M$ is the holonomy. The indispensable framework that is used to prove the required spectral bounds is known as Dolgopyat's method \cite{Dol98}, originally developed for the geodesic flow. The overarching idea is to use some form of a geometric property called the \emph{local non-integrability condition (LNIC)} to infer that the summands of the transfer operators are highly oscillating and thereby produce sufficient cancellations. As the name suggests, LNIC is related to the non-integrability of the stable and unstable foliations of the geodesic flow and consequently the first return time map.

In \cite{SW21}, to handle the frame flow for real hyperbolic manifolds, a more general LNIC was proven which takes into account not only the first return time map but also the holonomy. This is done by introducing a certain map which measures the $AM$-translation when going around a rectangle composed of stable and unstable foliations. Roughly speaking, the general LNIC says that given any direction $\omega \in \LieA \oplus \LieM$, the aforementioned map has a directional derivative close enough to $\omega$.

Due to the combination of the general LNIC and the delicate fractal nature of $\supp(\mathsf{m})$, a new property called the \emph{non-concentration property (NCP)} was also required. Roughly speaking, NCP says that the limit set $\limitset$ does not concentrate along any proper affine subspace: for any $x \in \limitset$, scale $\epsilon$, and direction $w$, one can find another point $y \in \limitset$ of distance $\asymp \epsilon$ such that the direction $y - x$ is close to $w$. If appropriate generalizations of these properties are proven, then they would imply exponential mixing for rank one locally symmetric spaces using Dolgopyat's method described above.

For nonreal hyperbolic manifolds, the key difference is that the associated root spaces do not have one but \emph{two} positive roots. Namely, there is an extra positive root which is double the simple root $\alpha$. As a result, we have the root spaces and horospherical Lie algebras
\begin{align*}
\LieSimple^+ &= \LieG_{-\alpha}, & \LieSimple^- &= \LieG_\alpha, & \LieN^+ &= \LieSimple^+ \oplus [\LieSimple^+, \LieSimple^+], &\LieN^- &= \LieSimple^- \oplus [\LieSimple^-, \LieSimple^-]
\end{align*}
which are mutually distinct. In contrast, $\LieN^+ = \LieSimple^+$ and $\LieN^- = \LieSimple^-$ for real hyperbolic manifolds. Thus, when we generalize the techniques from \cite{SW21} we must carefully choose between $(\LieN^+, \LieN^-)$ and $(\LieSimple^+, \LieSimple^-)$. Unfortunately, it turns out that a general NCP is not true on $\LieN^+$ because the limit set $\limitset$ \emph{does concentrate} along the subspace $\LieSimple^+ \subset \LieN^+$. This phenomenon comes from the expanding and contracting behavior of the root spaces at different exponential rates under the adjoint action of $A = \{a_t\}_{t \in \R}$. However, it suggests that a general NCP can be proven on $\LieSimple^+$ instead. In turn, we require a \emph{stronger form} for the general LNIC which looks only along $\LieSimple^+$ for the directional derivatives. This can be proven as in \cite{SW21} provided that the stronger property $[\LieSimple^+, \LieSimple^-] = \LieA \oplus \LieM$ rather than $[\LieN^+, \LieN^-] = \LieA \oplus \LieM$ holds. These two properties coincide for real hyperbolic manifolds and was proven in \cite{SW21}. Fortunately, we are able to prove the stronger property $[\LieSimple^+, \LieSimple^-] = \LieA \oplus \LieM$ by improving upon an argument due to Helgason \cite{Hel70}. Thus, we obtain the appropriate generalizations of LNIC and NCP.

We then demonstrate how the general LNIC and NCP are used in Dolgopyat's method to obtain the required cancellations. In doing so, we are careful in dealing with technicalities such as ensuring that the Patterson--Sullivan measure has the doubling property by using the Carnot--Carath\'{e}odory metric on $\LieN^+$, which was simply the Euclidean metric on $\mathbb R^{n - 1}$ for real hyperbolic manifolds in \cite{SW21}. This finishes the proof by offloading the remaining arguments to \cite{SW21}.

\subsection{Organization of the paper}
\label{subsec:OrganizationOfThePaper}
We start with covering the necessary background in \cref{sec:Preliminaries}. In \cref{sec:TransferOperatorsWithHolonomy} we introduce the transfer operators with holonomy and make a standard reduction of the main theorem done in Dolgopyat's method. \Cref{sec:LNIC&NCP} is the main section of this paper where we prepare for Dolgopyat's method by covering the key ingredients: LNIC and NCP. In \cref{sec:ConstructionOfDolgopyatOperators,sec:ProofOfFrameFlowDolgopyat}, we construct the Dolgopyat operators and go through Dolgopyat's method to obtain spectral bounds for large frequencies or nontrivial irreducible representations of $M$.

\subsection*{Acknowledgements}
We thank Hee Oh for her helpful suggestions for the manuscript.

\section{Preliminaries}
\label{sec:Preliminaries}
In this section, we cover the necessary background for the paper. We refer the reader to \cite[\S\S\ 2--4]{SW21} for more details for all subsections except \cref{subsec:LieTheory}.

Let us recall the classification of (irreducible) rank one symmetric spaces (of noncompact type). Both \cite[\S\ 19]{Mos73} and \cite{Hel01} are excellent references. There are three families: the real, complex, and quaternionic hyperbolic spaces which we denote by $\mathbb{H}_\R^n$, $\mathbb{H}_\C^n$, and $\mathbb{H}_\mathbb{H}^n$ for $n \geq 2$; and an exceptional one: the octonionic/Cayley hyperbolic plane which we denote by $\mathbb{H}_\mathbb{O}^2$. They are of real dimensions $\dim_\R(\mathbb{K})n$. Throughout the paper, any of the four algebras will be denoted by $\mathbb{K} \in \{\R, \C, \mathbb{H}, \mathbb{O}\}$ and in the last case we require that $n = 2$. Let $G = \Isom^+(\mathbb{H}_\mathbb{K}^n)$ be the group of orientation preserving isometries of $\mathbb{H}_\mathbb{K}^n$, which is then a noncompact connected simple Lie group of rank one. Fix a reference point $o \in \mathbb{H}_\mathbb{K}^n$, a reference tangent vector $v_o \in \T^1(\mathbb{H}_\mathbb{K}^n)$, and define the isotropy subgroups $K = \Stab_G(o)$ and $M = \Stab_G(v_o) < K$. \Cref{tab:RankOneSymmetricSpaces} shows what these groups are explicitly for each $\mathbb{K} \in \{\R, \C, \mathbb{H}, \mathbb{O}\}$. Let $\Gamma < G$ be a Zariski dense torsion-free discrete subgroup. Then our (irreducible) rank one locally symmetric space (of noncompact type) is $X = \Gamma \backslash \mathbb{H}_\mathbb{K}^n \cong \Gamma \backslash G/K$, its unit tangent bundle is $\T^1(X) \cong \Gamma \backslash G/M$, and its holonomy bundle in the sense of Kobayashi is $\Hol(X) \cong \Gamma \backslash G$. The holonomy bundle $\Hol(X)$ can be realized as a subbundle of the (oriented orthonormal) frame bundle $\F(X)$ and they coincide if and only if $\mathbb{K} = \R$. In fact, for $\mathbb{K} \in \{\R,\C,\mathbb H\}$, $\Hol(X)$ can be thought of as the $\mathbb{K}$-frame bundle, where $\mathbb{K}$-frames are modeled on $\mathbb{K}$-frames in $\mathbb{K}^n$ which are bases of mutually $\mathbb{K}$-orthogonal unit vectors $(e_1, e_2, \dotsc, e_n)$. Here, by $\mathbb{K}$-orthogonality, we mean that the column vectors satisfy $e_j^* e_k = \delta_{jk}$ for all $1 \leq j, k \leq n$, where $*$ denotes the $\mathbb{K}$-conjugate transpose with respect to the standard Hermitian form.

\begin{table}
\begin{tabular}{cccc}
\toprule
$\mathbb{K}$ & $G$ & $K$ & $M$ \\
\midrule
$\R$ & $\SO(n, 1)^\circ$ & $\SO(n)$ & $\SO(n - 1)$ \\
$\C$ & $\SU(n, 1)$ & $\Special(\U(n) \times \U(1))$ & $\Special(\U(n - 1) \times \U(1))$ \\
$\mathbb{H}$ & $\Sp(n, 1)$ & $\Sp(n) \times \Sp(1)$ & $\Sp(n - 1) \times \Sp(1)$ \\
$\mathbb{O}$ & $\ExceptionalLieF_4^{-20}$ & $\Spin(9)$ & $\Spin(7)$ \\
\bottomrule
\end{tabular}
\caption{Rank one symmetric spaces of noncompact type.}
\label{tab:RankOneSymmetricSpaces}
\end{table}

\subsection{Lie theory}
\label{subsec:LieTheory}
We again refer the reader to \cite[\S\ 19]{Mos73} and \cite{Hel01} for this subsection. Denote by $\LieG = \T_e(G)$ the Lie algebra of $G$. Let $B: \LieG \times \LieG \to \R$ be the Killing form. Let $\theta: \LieG \to \LieG$ be a Cartan involution, i.e., the symmetric bilinear form $B_\theta: \LieG \times \LieG \to \R$ defined by $B_\theta(x, y) = -B(x, \theta(y))$ for all $x, y \in \LieG$ is positive definite. Let $\LieG = \LieK \oplus \LieP$ be the associated eigenspace decomposition corresponding to the eigenvalues $+1$ and $-1$ respectively. Let $\LieA \subset \LieP$ be a maximal abelian subalgebra and $\Phi \subset \LieA^*$ be its restricted root system. Since $G$ is of rank one, $\dim(\LieA) = 1$. Let $\Phi^\pm \subset \Phi$ be a choice of sets of positive and negative roots. We have the associated restricted root space decomposition
\begin{align*}
\LieG = \LieA \oplus \LieM \oplus \LieN^+ \oplus \LieN^- = \LieA \oplus \LieM \oplus \bigoplus_{\alpha \in \Phi} \LieG_\alpha
\end{align*}
where $\LieM = Z_{\LieK}(\LieA) \subset \LieK$ and $\LieN^\pm = \bigoplus_{\alpha \in \Phi^\mp} \LieG_\alpha$. Let $\alpha \in \Phi^+$ be the simple root for the rest of this subsection. If $\mathbb{K} = \R$, then $\Phi^+ = \{\alpha\}$ and hence $\Phi$ is a reduced root system; and if $\mathbb{K} \in \{\C, \mathbb{H}, \mathbb{O}\}$, then $\Phi^+ = \{\alpha, 2\alpha\}$. We have the following proposition and its immediate corollary (see \cite[\S\ 23]{Mos73}, \cite[\S\ 4]{All99}, and \cite{Kim06}). Here $\overline{x}$ denotes the $\mathbb{K}$-conjugate and $\Im(x) = \frac{1}{2}(x - \overline{x})$ denotes the imaginary component of $x \in \mathbb{K}$.

\begin{proposition}
We have the isomorphisms of graded Lie algebras
\begin{align*}
\LieG_\alpha \oplus \LieG_{2\alpha} \cong \LieG_{-\alpha} \oplus \LieG_{-2\alpha} \cong \mathbb{K}^{n - 1} \oplus \Im\mathbb{K}
\end{align*}
where the latter is endowed with a Lie bracket characterized by:
\begin{enumerate}
\item $[\Im\mathbb{K}, \mathbb{K}^{n - 1} \oplus \Im\mathbb{K}] = 0$,
\item $[x, y] = 2\Im\langle x, y\rangle = 2\Im\bigl(\sum_{j = 1}^{n - 1} x_j \overline{y_j}\bigr) \in \Im\mathbb{K}$ for all $x, y \in \mathbb{K}^{n - 1}$.
\end{enumerate}
\end{proposition}

\begin{corollary}
We have $\LieG_{2\alpha} = [\LieG_\alpha, \LieG_\alpha]$ and $\LieG_{-2\alpha} = [\LieG_{-\alpha}, \LieG_{-\alpha}]$.
\end{corollary}

We denote $\LieSimple^+ = \LieG_{-\alpha}$ and $\LieSimple^- = \LieG_\alpha$. Then for any $\mathbb{K} \in \{\R, \C, \mathbb{H}, \mathbb{O}\}$, we have $\LieN^+ = \LieSimple^+ \oplus [\LieSimple^+, \LieSimple^+]$ and $\LieN^- = \LieSimple^- \oplus [\LieSimple^-, \LieSimple^-]$. If $\mathbb{K} = \R$, then $[\LieSimple^+, \LieSimple^+]$ and $[\LieSimple^-, \LieSimple^-]$ are trivial and hence $\LieN^+$ and $\LieN^-$ are abelian; and if $\mathbb{K} \in \{\C, \mathbb{H}, \mathbb{O}\}$, then $\LieN^+$ and $\LieN^-$ are 2-step nilpotent.

Define the Lie subgroups of $G$ by
\begin{align*}
A &= \exp(\LieA), & N^\pm &= \exp(\LieN^\pm).
\end{align*}
Let $K < G$ be the maximal compact subgroup whose Lie algebra is $\LieK$. Let $M = C_K(A)$ which is then connected (see \cite[Lemma 2.3(1)]{Win15}) and whose Lie algebra is $\LieM$. Note that $C_G(A) = AM$. Define $N = N^-$ and $P = MAN$.

We fix a left $G$-invariant and right $K$-invariant Riemannian metric on $G$ and denote the corresponding inner product and norm on any of its tangent spaces by $\langle \cdot, \cdot \rangle$ and $\|\cdot\|$ respectively. This induces a left $G$-invariant and right $K$-invariant metric $d$ on $G$. We use the same notations for inner products, norms, and metrics induced on any other quotient spaces.

The Killing form $B$ of $\LieG$ is $\Ad_G$-invariant, and hence $\Ad_K$-invariant. Moreover, $B|_{\LieA \times \LieA}$ is positive definite since $\theta$ is a Cartan involution. Thus, $\langle \cdot, \cdot \rangle|_{\LieA \times \LieA} \asymp B|_{\LieA \times \LieA}$. We also assume that the Riemannian metric on $G$ is scaled such that identifying $\LieA \cong \R$ as inner product spaces, we have $\alpha(1) = 1$ for the simple root $\alpha \in \Phi^+$. We obtain a corresponding parametrization $A = \{a_t\}_{t \in \R}$ such that its right translation action corresponds to the geodesic flow on $\Gamma \backslash G/M$ and the frame flow on $\Gamma \backslash G$. Using these conventions and the identity $\Ad \circ \exp = \exp \circ \ad$, we have the useful calculation that
\begin{align*}
\Ad_{a_t}(x_{\LieSimple^\pm} + x_{[\LieSimple^\pm, \LieSimple^\pm]}) = e^{\mp t}x_{\LieSimple^\pm} + e^{\mp2t}x_{[\LieSimple^\pm, \LieSimple^\pm]}
\end{align*}
for all $x_{\LieSimple^\pm} + x_{[\LieSimple^\pm, \LieSimple^\pm]} \in \LieSimple^\pm \oplus [\LieSimple^\pm, \LieSimple^\pm]$ and $t \in \R$.

\subsection{Convex cocompact subgroups}
Let $\partial_\infty\mathbb{H}_\mathbb{K}^n$ and $\overline{\mathbb{H}_\mathbb{K}^n} = \mathbb{H}_\mathbb{K}^n \cup \partial_\infty\mathbb{H}_\mathbb{K}^n$ be the boundary at infinity and the compactification of $\mathbb{H}_\mathbb{K}^n$ respectively. We have the isomorphism $\partial_\infty\mathbb{H}_\mathbb{K}^n \cong G/P$ as $G$-spaces. Let $\limitset = \lim(\Gamma o) \subset \partial_\infty\mathbb{H}_\mathbb{K}^n \subset \overline{\mathbb{H}_\mathbb{K}^n}$ be the limit set of $\Gamma$. We denote its convex hull by $\Hull(\limitset) \subset \mathbb{H}_\mathbb{K}^n$. The critical exponent $\delta_\Gamma$ of $\Gamma$ is the abscissa of convergence of the Poincar\'{e} series $\mathscr{P}_\Gamma(s) = \sum_{\gamma \in \Gamma} e^{-s d(o, \gamma o)}$. It is well known that these are independent of the choice of $o \in \mathbb{H}_\mathbb{K}^n$.

\begin{definition}[Convex cocompact subgroup]
A discrete subgroup $\Gamma < G$ is called a \emph{convex cocompact subgroup} if $\Core(X) := \Gamma \backslash \Hull(\limitset) \subset X$ is compact.
\end{definition}

We assume that $\Gamma$ is convex cocompact in the entire paper.

\subsection{Patterson--Sullivan density}
\label{subsec:Patterson--SullivanDensity}
Let $\beta: \partial_\infty\mathbb{H}_\mathbb{K}^n \times \mathbb{H}_\mathbb{K}^n \times \mathbb{H}_\mathbb{K}^n \to \R$ denote the \emph{Busemann function} defined by $\beta_{\xi}(y, x) = \lim_{t \to \infty} (d(\xi(t), y) - d(\xi(t), x))$ for all $\xi \in \partial_\infty\mathbb{H}_\mathbb{K}^n$ and $x, y \in \mathbb{H}_\mathbb{K}^n$ where $\xi: \mathbb R \to \mathbb{H}_\mathbb{K}^n$ is any geodesic with $\lim_{t \to \infty} \xi(t) = \xi$. For convenience, we allow tangent vector arguments for the Busemann function as well in which case we will use their basepoints in the definition.

Let $\{\mu^{\mathrm{PS}}_x: x \in \mathbb{H}_\mathbb{K}^n\}$ be the \emph{Patterson--Sullivan density} of $\Gamma$ \cite{Pat76,Sul79,CI99}, i.e., the set of finite Borel measures on $\partial_\infty\mathbb{H}_\mathbb{K}^n$ supported on $\limitset$ such that
\begin{enumerate}
\item	$\gamma_*\mu^{\mathrm{PS}}_x = \mu^{\mathrm{PS}}_{\gamma x}$ for all $\gamma \in \Gamma$ and $x \in \mathbb{H}_\mathbb{K}^n$,
\item	$\frac{d\mu^{\mathrm{PS}}_x}{d\mu^{\mathrm{PS}}_y}(\xi) = e^{\delta_\Gamma \beta_{\xi}(y, x)}$ for all $\xi \in \partial_\infty\mathbb{H}_\mathbb{K}^n$ and $x, y \in \mathbb{H}_\mathbb{K}^n$.
\end{enumerate}
Since $\Gamma$ is convex cocompact, for all $x \in \mathbb{H}_\mathbb{K}^n$, the measure $\mu^{\mathrm{PS}}_x$ is the $\delta_\Gamma$-dimensional Hausdorff measure on $\partial_\infty\mathbb{H}_\mathbb{K}^n$ supported on $\limitset$ corresponding to the spherical metric on $\partial_\infty\mathbb{H}_\mathbb{K}^n$ with respect to $x$, up to scalar multiples.

\subsection{Bowen--Margulis--Sullivan measure}
\label{subsec:BMS_Measure}
For all $u \in \T^1(\mathbb{H}_\mathbb{K}^n)$, its forward and backward limit points are $u^\pm = \lim_{t \to \pm\infty} \gamma(t) \in \partial_\infty\mathbb{H}_\mathbb{K}^n$ where $\gamma: \R \to \mathbb{H}_\mathbb{K}^n$ is the geodesic with $\gamma'(0) = u$. Similarly, for all $f = (e_1, e_2, \dotsc, e_{\dim_\R(\mathbb{K})n}) \in \F(\mathbb{H}_\mathbb{K}^n)$, we write $f^\pm = e_1^\pm$. Since $\Hol(\mathbb{H}_\mathbb{K}^n) \subset \F(\mathbb{H}_\mathbb{K}^n)$ is a subbundle, this notation makes sense for all $g \in \Hol(\mathbb{H}_\mathbb{K}^n) \cong G$ as well. Using the isomorphism $\partial_\infty\mathbb{H}_\mathbb{K}^n \cong G/P$ as $G$-spaces, we have the algebraic description $g^+ = gP$ and $g^- = gw_0P$ for all $g \in G$, where $w_0 \in N_K(A)/M \cong \mathbb{Z}/2\mathbb{Z}$ is the nontrivial element of the Weyl group.

Using the Hopf parametrization via the diffeomorphism $G/M \cong \T^1(\mathbb{H}_\mathbb{K}^n) \to \{(u^+, u^-) \in \partial_\infty\mathbb{H}_\mathbb{K}^n \times \partial_\infty\mathbb{H}_\mathbb{K}^n: u^+ \neq u^-\} \times \mathbb R$ defined by $u \mapsto (u^+, u^-, t = \beta_{u^-}(o, u))$, we define the \emph{Bowen--Margulis--Sullivan (BMS) measure} $\mathsf{m}$ on $G/M$ \cite{Mar04,Bow71,Kai90,CI99} by
\begin{align*}
d\mathsf{m}(u) = e^{\delta_\Gamma \beta_{u^+}(o, u)} e^{\delta_\Gamma \beta_{u^-}(o, u)} \, d\mu^{\mathrm{PS}}_o(u^+) \, d\mu^{\mathrm{PS}}_o(u^-) \, dt.
\end{align*}
This definition depends only on $\Gamma$. Moreover, $\mathsf{m}$ is left $\Gamma$-invariant. We now define induced measures on other spaces, all of which we call the BMS measures and denote by $\mathsf{m}$ by abuse of notation. By left $\Gamma$-invariance, $\mathsf{m}$ descends to a measure on $\Gamma \backslash G/M$. We normalize it to a probability measure so that $\mathsf{m}(\Gamma \backslash G/M) = 1$. Since $M$ is compact, we can then use the Haar probability measure on $M$ to lift $\mathsf{m}$ to a right $M$-invariant measure on $\Gamma \backslash G$. The BMS measures are right $A$-invariant and hence invariant under the geodesic and frame flows. Define $\Omega = \supp(\mathsf{m}) \subset \Gamma \backslash G/M$ which is right $A$-invariant and also compact since $\Gamma$ is convex cocompact.

\subsection{Markov section}
\label{subsec:MarkovSections}
Recall that the geodesic flow on $\T^1(X)$ is Anosov because rank one locally symmetric spaces are of pinched negative sectional curvature. By the independent works of Bowen \cite{Bow70} and Ratner \cite{Rat73} which was later generalized by Pollicott \cite{Pol87}, there exists a Markov section for the geodesic flow on $\Omega \subset \T^1(X) \cong \Gamma \backslash G/M$ which we define below.

Let $W^{\mathrm{su}}(w), W^{\mathrm{ss}}(w) \subset \T^1(X)$ and $W^{\mathrm{ss}}(w) \subset \T^1(X)$ denote the leaves through $w \in \T^1(X)$ of the strong unstable and strong stable foliations, and $W_{\epsilon}^{\mathrm{su}}(w) \subset W^{\mathrm{su}}(w)$ and $W_{\epsilon}^{\mathrm{ss}}(w) \subset W^{\mathrm{ss}}(w)$ denote the open balls of radius $\epsilon > 0$ with respect to the induced distance functions $d_{\mathrm{su}}$ and $d_{\mathrm{ss}}$ respectively. We use similar notations for the weak unstable and weak stable foliations by replacing `su' with `wu' and `ss' with `ws' respectively. There exist $\epsilon_0, \epsilon_0' > 0$ such that for all $w \in \T^1(X)$, $u \in W_{\epsilon_0}^{\mathrm{wu}}(w)$, and $s \in W_{\epsilon_0}^{\mathrm{ss}}(w)$, there exists a unique intersection denoted by
\begin{align}
\label{eqn:BracketOfUandS}
[u, s] = W_{\epsilon_0'}^{\mathrm{ss}}(u) \cap W_{\epsilon_0'}^{\mathrm{wu}}(s)
\end{align}
and moreover, $[\cdot, \cdot]$ defines a homeomorphism from $W_{\epsilon_0}^{\mathrm{wu}}(w) \times W_{\epsilon_0}^{\mathrm{ss}}(w)$ onto its image \cite{Rat73}. For any $\hat{\delta} > 0$ and \emph{center} $w \in \Omega$, we call
\begin{align*}
R = [U, S] = \{[u, s] \in \Omega: u \in U, s \in S\} \subset \Omega
\end{align*}
a \emph{rectangle of size $\hat{\delta}$} if $\diam_{d_{\mathrm{su}}}(U), \diam_{d_{\mathrm{ss}}}(S) \leq \hat{\delta}$, and $w \in U \subset W_{\epsilon_0}^{\mathrm{su}}(w) \cap \Omega$ and $w \in S \subset W_{\epsilon_0}^{\mathrm{ss}}(w) \cap \Omega$ are \emph{proper} subsets, i.e., $U = \overline{\interior(U)}$ and $S = \overline{\interior(S)}$ with respect to the topologies of $W^{\mathrm{su}}(w) \cap \Omega$ and $W^{\mathrm{ss}}(w) \cap \Omega$ respectively. For any rectangle $R = [U, S]$, we generalize the notation and define $[v_1, v_2] = [u_1, s_2]$ for all $v_1 = [u_1, s_1] \in R$ and $v_2 = [u_2, s_2] \in R$.

Let $\mathcal{R} = \{R_1 = [U_1, S_1], R_2 = [U_2, S_2], \dotsc, R_N = [U_N, S_N]\}$ for some $N \in \N$ be a set of mutually disjoint rectangles of size $\hat{\delta}$ in $\Omega$. It is said to be \emph{complete} if $\Omega = \bigcup_{j = 1}^N \bigcup_{t \in [0, \hat{\delta}]} R_j a_t$. Denote $R = \bigsqcup_{j = 1}^N R_j$ and $U = \bigsqcup_{j = 1}^N U_j$. The first return time map $\tau: R \to \mathbb R$ and the Poincar\'{e} first return map $\mathcal{P}: R \to R$ are defined by
\begin{align*}
\tau(u) &= \inf\{t > 0: ua_t \in R\}, & \mathcal{P}(u) &= ua_{\tau(u)}
\end{align*}
for all $u \in R$. Let $\sigma = (\proj_U \circ \mathcal{P})|_U: U \to U$ where $\proj_U: R \to U$ is the projection defined by $\proj_U([u, s]) = u$ for all $[u, s] \in R$.

\begin{definition}[Markov section]
The complete set $\mathcal{R}$ is called a \emph{Markov section} if the following \emph{Markov property} is satisfied: $[\interior(U_k), \mathcal{P}(u)] \subset \mathcal{P}([\interior(U_j), u])$ and $\mathcal{P}([u, \interior(S_j)]) \subset [\mathcal{P}(u), \interior(S_k)]$ for all $u \in R$ such that $u \in \interior(R_j) \cap \mathcal{P}^{-1}(\interior(R_k)) \neq \varnothing$ and $1 \leq j, k \leq N$.
\end{definition}

Henceforth, we fix any positive constant
\begin{align}
\label{eqn:DeltaHatCondition}
\hat{\delta} < \min(1, \epsilon_0/2, \epsilon_0'/2),
\end{align}
where $\epsilon_0$ and $\epsilon_0'$ are from \cref{eqn:BracketOfUandS}, and any corresponding Markov section $\mathcal{R}$. We introduce the distance function $d$ on $U$ defined by
\begin{align*}
d(u, v) =
\begin{cases}
d_{\mathrm{su}}(u, v), & u, v \in U_j \text{ for some } 1 \leq j \leq N \\
1, & \text{otherwise}.
\end{cases}
\end{align*}

\subsection{Symbolic dynamics}
\label{subsec:SymbolicDynamics}
Let $\mathcal A = \{1, 2, \dotsc, N\}$ be the \emph{alphabet} for the coding corresponding to the Markov section $\mathcal{R}$ and define
\begin{align*}
\Sigma &= \{(\dotsc, x_{-1}, x_0, x_1, \dotsc) \in \mathcal A^{\mathbb Z}: T_{x_j, x_{j + 1}} = 1 \text{ for all } j \in \mathbb Z\}, \\
\Sigma^+ &= \{(x_0, x_1, \dotsc) \in \mathcal A^{\mathbb Z_{\geq 0}}: T_{x_j, x_{j + 1}} = 1 \text{ for all } j \in \mathbb Z_{\geq 0}\}.
\end{align*}
Sequences in the above spaces and their finite counterparts are said to be \emph{admissible}. For any $\kappa \in (0, 1)$, we can endow $\Sigma$ (resp. $\Sigma^+$) with the distance function $d_\kappa$ defined by $d_\kappa(x, y) = \kappa^{\inf\{|j| \in \mathbb Z_{\geq 0}: x_j \neq y_j\}}$ for all $x, y \in \Sigma$ (resp $x, y \in \Sigma^+$).

\begin{definition}[Cylinder]
For all $k \in \mathbb Z_{\geq 0}$ and for all admissible sequences $x = (x_0, x_1, \dotsc, x_k)$, we define the corresponding \emph{cylinder} to be
\begin{align*}
\mathtt{C}[x] = \{u \in U: \sigma^j(u) \in \interior(U_{x_j}) \text{ for all } 0 \leq j \leq k\}
\end{align*}
with \emph{length} $\len(\mathtt{C}[x]) := \len(x) := k$.
\end{definition}

Although $\sigma$ and $\tau$ are not continuous, in our setting, $\sigma|_{\mathtt{C}[j, k]}$ is bi-Lipschitz and $\tau|_{\mathtt{C}[j, k]}$ is Lipschitz for all admissible pairs $(j, k)$.

Abusing notation, let $\sigma$ denote the shift map on $\Sigma$ or $\Sigma^+$. There are continuous surjections $\zeta: \Sigma \to R$ and $\zeta^+: \Sigma^+ \to U$ defined by $\zeta(x) = \bigcap_{j = -\infty}^\infty \overline{\mathcal{P}^{-j}(\interior(R_{x_j}))}$ for all $x \in \Sigma$ and $\zeta^+(x) = \bigcap_{j = 0}^\infty \overline{\sigma^{-j}(\interior(U_{x_j}))}$ for all $x \in \Sigma^+$. Fix $\kappa \in (0, 1)$ sufficiently close to $1$ such that $\zeta$ and $\zeta^+$ are Lipschitz \cite[Lemma 2.2]{Bow73}. Let $C^{\Lip(d_\kappa)}(\Sigma, \mathbb R)$ denote the space of Lipschitz functions $f: \Sigma \to \mathbb R$ and similarly for domain space $\Sigma^+$ or target space $\mathbb C$. Since $(\tau \circ \zeta)|_{\zeta^{-1}(\proj_U^{-1}(\mathtt{C}))}$ and $(\tau \circ \zeta^+)|_{(\zeta^+)^{-1}(\mathtt{C})}$ are Lipschitz for cylinders $\mathtt{C} \subset U$ of length $1$, they have Lipschitz extensions $\tau_\Sigma: \Sigma \to \mathbb R$ and $\tau_{\Sigma^+}: \Sigma^+ \to \mathbb R$ respectively. Note that $\zeta(\sigma(x)) = \zeta(x)a_{\tau_\Sigma(x)}$ for all $x \in \Sigma$ and $\zeta^+(\sigma(x)) = \proj_U(\zeta^+(x)a_{\tau_{\Sigma^+}(x)})$ for all $x \in \Sigma^+$.

\subsection{Thermodynamics}
\label{subsec:Thermodynamics}
\begin{definition}[Pressure]
\label{def:Pressure}
For all $f \in C^{\Lip(d_\kappa)}(\Sigma, \mathbb R)$, called the \emph{potential}, the \emph{pressure} is defined by
\begin{align*}
\Pr_\sigma(f) = \sup_{\nu \in \mathcal{M}^1_\sigma(\Sigma)}\left\{\int_\Sigma f \, d\nu + h_\nu(\sigma)\right\}
\end{align*}
where $\mathcal{M}^1_\sigma(\Sigma)$ is the set of $\sigma$-invariant Borel probability measures on $\Sigma$ and $h_\nu(\sigma)$ is the measure theoretic entropy of $\sigma$ with respect to $\nu$.
\end{definition}

The supremum in \cref{def:Pressure} is attained by a unique $\nu_f \in \mathcal{M}^1_\sigma(\Sigma)$ \cite[Theorems 2.17 and 2.20]{Bow08}. Define $\nu_\Sigma = \nu_{-\delta_\Gamma\tau_\Sigma} \in \mathcal{M}^1_\sigma(\Sigma)$ whose corresponding pressure is $\Pr_\sigma(-\delta_\Gamma\tau_\Sigma) = 0$. Define $\nu_R = \zeta_*(\nu_\Sigma)$ and $\nu_U = (\proj_U)_*(\nu_R)$. Note that $\nu_U(\tau) = \nu_R(\tau)$. For the relation between $\mathsf{m}$ and $\nu_R$, which we do not require directly in this paper, see \cite[\S\S\ 3 and 4]{SW21}.

\subsection{Holonomy}
Let $w_j \in R_j$ be the center for all $j \in \mathcal{A}$. A smooth section
\begin{align*}
F: \bigsqcup_{j = 1}^N [W_{\epsilon_0}^{\mathrm{su}}(w_j), W_{\epsilon_0}^{\mathrm{ss}}(w_j)] \to \Hol(X) \cong \Gamma \backslash G
\end{align*}
can be constructed with the properties that for all $j \in \mathcal{A}$, and $u, u' \in W_{\epsilon_0}^{\mathrm{su}}(w_j)$, and $s, s' \in W_{\epsilon_0}^{\mathrm{ss}}(w_j)$, we have
\begin{align*}
\lim_{t \to -\infty} d(F(u)a_t, F(u')a_t) &= 0, & \lim_{t \to +\infty} d(F([u, s])a_t, F([u, s'])a_t) &= 0
\end{align*}
or equivalently, there exist unique $n^+ \in N^+$ and $n^- \in N^-$ such that
\begin{align*}
F(u') &= F(u)n^+, & F([u, s']) &= F([u, s])n^-.
\end{align*}
We refer to \cite[\S\ 4]{SW21} for the details; the only difference here being that we generalize $F$ as being $\Hol(X)$-valued rather than $\F(X)$-valued.

\begin{definition}[Holonomy]
The \emph{holonomy} is a map $\vartheta: R \to M$ such that for all $u \in R$, we have $F(\mathcal{P}(u)) = F(u) a_{\tau(u)} \vartheta(u)$.
\end{definition}

Just as in the observation in \cite[Subsection 3.1]{SW21} and \cite[Lemma 4.2]{SW21}, we have the following.

\begin{lemma}
\label{lem:tauAndHolonomyConstantOnStrongStableLeaves}
The maps $\tau$ and $\vartheta$ are constant on $[u, S_j]$ for all $u \in U_j$ and $j \in \mathcal{A}$.
\end{lemma}

Denote the unitary dual of $M$ by $\widehat{M}$. Denote the trivial irreducible representation by $1 \in \widehat{M}$ and define $\widehat{M}_0 = \widehat{M} \setminus \{1\}$. By the Peter--Weyl theorem, we obtain an orthogonal Hilbert space decomposition
\begin{align*}
L^2(M, \mathbb C) = \operatorname*{\widehat{\bigoplus}}_{\rho \in \widehat{M}} V_\rho^{\oplus \dim(\rho)}.
\end{align*}

For all $b \in \mathbb R$ and $\rho \in \widehat{M}$, we define the tensored unitary representation $\rho_b: AM \to \U(V_\rho)$ by
\begin{align*}
\rho_b(a_tm)(z) = e^{-ibt}\rho(m)(z) \qquad \text{for all $z \in V_\rho$, $t \in \mathbb R$, and $m \in M$}.
\end{align*}
For any representation $\rho: M \to \U(V)$ for some Hilbert space $V$, we denote the differential at $e \in M$ by $d\rho = (d\rho)_e: \LieM \to \mathfrak{u}(V)$, and define the norm
\begin{align*}
\|\rho\| = \sup_{z \in \LieM, \|z\| = 1} \|d\rho(z)\|_{\mathrm{op}}
\end{align*}
and similarly for any unitary representation $\rho: AM \to \U(V)$. We have the following facts (see \cite[Lemmas 4.3 and 4.4]{SW21} for proofs).

\begin{lemma}
\label{lem:LieTheoreticNormBounds}
For all $b \in \mathbb R$ and $\rho \in \widehat{M}$, we have $\max(|b|, \|\rho\|) \leq \|\rho_b\| \leq |b| + \|\rho\|$.
\end{lemma}

Dolgopyat's method is concerning the regime when $\|\rho_b\|$ is sufficiently large. This occurs precisely when $|b|$ is sufficiently large or $\rho \in \widehat{M}$ is nontrival. Thus, we define
\begin{align*}
\widehat{M}_0(b_0) = \{(b, \rho) \in \mathbb R \times \widehat{M}: |b| > b_0 \text{ or } \rho \neq 1\}
\end{align*}
where we fix $b_0 > 0$ later. Fix $\delta_{\varrho} = \inf_{b \in \mathbb R, \rho \in \widehat{M}_0} \|\rho_b\| = \inf_{\rho \in \widehat{M}_0} \|\rho\|$ which is positive because $M$ is a compact connected Lie group. Also fix $\delta_{1, \varrho} = \min(1, \delta_{\varrho})$.

\begin{lemma}
\label{lem:maActionLowerBound}
There exists $\varepsilon_1 > 0$ such that for all $b \in \mathbb R$, $\rho \in \widehat{M}$, and $\omega \in V_\rho^{\oplus \dim(\rho)}$ with $\|\omega\|_2 = 1$, there exists $z \in \LieA \oplus \LieM$ with $\|z\| = 1$ such that $\|d\rho_b(z)(\omega)\|_2 \geq \varepsilon_1 \|\rho_b\|$.
\end{lemma}

\subsection{Extending to smooth maps}
\label{subsec:ExtendingToSmoothMaps}
We need to use the smooth structure on $G$ to apply Lie theoretic arguments later on. Thanks to \cite{Rue89} (using an adapted metric), there exist open neighborhoods $\tilde{U}_j \supset U_j$ with $\overline{\tilde{U}_j} \subset W_{\epsilon_0}^{\mathrm{su}}(w_j)$ and $\diam_{d_{\mathrm{su}}}(\tilde{U}_j) \le \hat{\delta}$ such that for all admissible pairs $(j, k)$, there exists a canonical extension of $(\sigma|_{\mathtt{C}[j, k]})^{-1}: \interior(U_k) \to \mathtt{C}[j, k]$ to a map $\sigma^{-(j, k)}: \tilde{U}_k \to \tilde{U}_j$ which is a diffeomorphism onto its image. Define $\tilde{R}_j = [\tilde{U}_j, S_j]$ for all $j \in \mathcal{A}$. Define $\tilde{R} = \bigsqcup_{j = 1}^N \tilde{R}_j$, $\tilde{U} = \bigsqcup_{j = 1}^N \tilde{U}_j$ and the measure $\nu_{\tilde{U}}$ on $\tilde{U}$ supported on $U$ with $\nu_{\tilde{U}}\bigr|_U = \nu_U$. Let $j \in \mathbb Z_{\geq 0}$ and $\alpha = (\alpha_0, \alpha_1, \dotsc, \alpha_j)$ be an admissible sequence. Define $\sigma^{-\alpha} = \sigma^{-(\alpha_0, \alpha_1)} \circ \sigma^{-(\alpha_1, \alpha_2)} \circ \cdots \circ \sigma^{-(\alpha_{j - 1}, \alpha_j)}$ if $j > 0$ and $\sigma^{-\alpha} = \Id_{\tilde{U}_{\alpha_0}}$ if $j = 0$. Define the cylinder $\tilde{\mathtt{C}}[\alpha] = \sigma^{-\alpha}(\tilde{U}_{\alpha_j}) \supset \mathtt{C}[\alpha]$. Define the smooth maps $\sigma^\alpha = (\sigma^{-\alpha})^{-1}: \tilde{\mathtt{C}}[\alpha] \to \tilde{U}_{\alpha_j}$.

For all admissible pairs $(j, k)$, the maps $\tau|_{\mathtt{C}[j, k]}$ and $\vartheta|_{\mathtt{C}[j, k]}$ naturally extend to smooth maps $\tau_{(j, k)}: \tilde{\mathtt{C}}[j, k] \to \mathbb R$ and $\vartheta^{(j, k)}: \tilde{\mathtt{C}}[j, k] \to M$. Now, for all $k \in \mathbb N$ and admissible sequences $\alpha = (\alpha_0, \alpha_1, \dotsc, \alpha_k)$, we define the smooth maps $\tau_\alpha: \tilde{\mathtt{C}}[\alpha] \to \mathbb R$, $\vartheta^\alpha: \tilde{\mathtt{C}}[\alpha] \to M$, and $\Phi^\alpha: \tilde{\mathtt{C}}[\alpha] \to AM$ by
\begin{align*}
\tau_\alpha(u) &= \sum_{j = 0}^{k - 1} \tau_{(\alpha_j, \alpha_{j + 1})}(\sigma^{(\alpha_0, \alpha_1, \dotsc, \alpha_j)}(u)), \\
\vartheta^\alpha(u) &= \prod_{j = 0}^{k - 1} \vartheta^{(\alpha_j, \alpha_{j + 1})}(\sigma^{(\alpha_0, \alpha_1, \dotsc, \alpha_j)}(u)), \\
\Phi^\alpha(u) &= a_{\tau_\alpha(u)}\vartheta^\alpha(u) = \prod_{j = 0}^{k - 1} \Phi^{(\alpha_j, \alpha_{j + 1})}(\sigma^{(\alpha_0, \alpha_1, \dotsc, \alpha_j)}(u))
\end{align*}
for all $u \in \tilde{\mathtt{C}}[\alpha]$, where the terms of the products are to be in \emph{ascending} order from left to right. For all admissible sequences $\alpha$ with $\len(\alpha) = 0$, we define $\tau_\alpha(u) = 0$ and $\vartheta^\alpha(u) = \Phi^\alpha(u) = e \in AM$ for all $u \in \tilde{\mathtt{C}}[\alpha]$.

\section{Transfer operators with holonomy, their spectral bounds, and the proof of \texorpdfstring{\cref{thm:TheoremExponentialMixingOfFrameFlow}}{\autoref{thm:TheoremExponentialMixingOfFrameFlow}}}
\label{sec:TransferOperatorsWithHolonomy}
\subsection{Transfer operators}
\label{subsec:TransferOperators}
We will use the notation $\xi = a + ib \in \mathbb C$ for the complex parameter for the transfer operators and use the convention that sums over sequences are actually sums over \emph{admissible} sequences, throughout the paper.

\begin{definition}[Transfer operator with holonomy]
For all $\xi \in \mathbb C$ and $\rho \in \widehat{M}$, the \emph{transfer operator with holonomy} $\tilde{\mathcal{M}}_{\xi\tau, \rho}: C\bigl(\tilde{U}, V_\rho^{\oplus \dim(\rho)}\bigr) \to C\bigl(\tilde{U}, V_\rho^{\oplus \dim(\rho)}\bigr)$ is defined by
\begin{align*}
\tilde{\mathcal{M}}_{\xi\tau, \rho}(H)(u) = \sum_{\substack{(j, k)\\ u' = \sigma^{-(j, k)}(u)}} e^{\xi\tau_{(j, k)}(u')} \rho(\vartheta^{(j, k)}(u')^{-1}) H(u')
\end{align*}
for all $u \in \tilde{U}$ and $H \in C\bigl(\tilde{U}, V_\rho^{\oplus \dim(\rho)}\bigr)$.
\end{definition}

Let $\xi \in \mathbb C$ and $\rho \in \widehat{M}$. We define $\mathcal{M}_{\xi\tau, \rho}: C\bigl(U, V_\rho^{\oplus \dim(\rho)}\bigr) \to C\bigl(U, V_\rho^{\oplus \dim(\rho)}\bigr)$ in a similar fashion. We simply call $\tilde{\mathcal{L}}_{\xi\tau} := \tilde{\mathcal{M}}_{\xi\tau, 1}$ and $\mathcal{L}_{\xi\tau} := \mathcal{M}_{\xi\tau, 1}$ \emph{transfer operators}.

Let $a \in \mathbb R$. Recalling the Ruelle--Perron--Frobenius (RPF) theorem and the theory of Gibbs measures (see \cite{Bow08,PP90}), there exist a unique positive function $h_a \in C^{\Lip(d)}(U, \mathbb R)$ and a measure $\nu_a$ on $U$ such that $\nu_a(h_a) = 1$ and
\begin{align*}
\mathcal{L}_{-(\delta_\Gamma + a)\tau}(h_a) &= \lambda_a h_a, & \mathcal{L}_{-(\delta_\Gamma + a)\tau}^*(\nu_a) &= \lambda_a \nu_a
\end{align*}
where $\lambda_a := e^{\Pr_\sigma(-(\delta_\Gamma + a)\tau_{\Sigma})}$ is the maximal simple eigenvalue of $\mathcal{L}_{-(\delta_\Gamma + a)\tau}$ and the rest of the spectrum of $\mathcal{L}_{-(\delta_\Gamma + a)\tau}|_{C^{\Lip(d)}(U, \mathbb C)}$ is contained in a disk of radius strictly less than $\lambda_a$. Moreover, $d\nu_U = h_0 \, d\nu_0$ and $\lambda_0 = 1$ (see \cref{subsec:Thermodynamics}). By \cite[Theorem A.2]{SW21}, which also holds in our setting, the eigenvector $h_a \in C^{\Lip(d)}(U, \mathbb R)$ extends to an eigenvector $h_a \in C^\infty(\tilde{U}, \mathbb R)$ for $\tilde{\mathcal{L}}_{-(\delta_\Gamma + a)\tau}$ with bounded derivatives.

In light of the above, it is convenient to normalize the transfer operators with holonomy. Let $a \in \mathbb R$. For all admissible pairs $(j, k)$, we define the smooth map
\begin{align*}
f_{(j, k)}^{(a)} = - \log(\lambda_a) - (\delta_\Gamma + a)\tau_{(j, k)} + \log \circ h_0 - \log \circ h_0 \circ \sigma^{(j, k)}.
\end{align*}
For all $k \in \mathbb Z_{\geq 0}$ and admissible sequences $\alpha = (\alpha_0, \alpha_1, \dotsc, \alpha_k)$, we define the smooth map $f_\alpha^{(a)}: \tilde{\mathtt{C}}[\alpha] \to \mathbb R$ similar to $\tau_\alpha$ in \cref{subsec:ExtendingToSmoothMaps}. Let $\xi \in \mathbb C$ and $\rho \in \widehat{M}$. We define $\tilde{\mathcal{M}}_{\xi, \rho}: C\bigl(\tilde{U}, V_\rho^{\oplus \dim(\rho)}\bigr) \to C\bigl(\tilde{U}, V_\rho^{\oplus \dim(\rho)}\bigr)$ by
\begin{align*}
\tilde{\mathcal{M}}_{\xi, \rho}(H)(u) = \sum_{\substack{(j, k)\\ u' = \sigma^{-(j, k)}(u)}} e^{(f_{(j, k)}^{(a)} + ib\tau_{(j, k)})(u')} \rho(\vartheta^{(j, k)}(u')^{-1}) H(u')
\end{align*}
and for all $k \in \mathbb N$, its $k$\textsuperscript{th} iteration is
\begin{align}
\label{eqn:k^thIterationOfCongruenceTransferOperatorOfType_rho}
\tilde{\mathcal{M}}_{\xi, \rho}^k(H)(u) = \sum_{\substack{\alpha: \len(\alpha) = k\\ u' = \sigma^{-\alpha}(u)}} e^{f_\alpha^{(a)}(u')} \rho_b(\Phi^\alpha(u')^{-1}) H(u')
\end{align}
for all $u \in \tilde{U}$ and $H \in C\bigl(\tilde{U}, V_\rho^{\oplus \dim(\rho)}\bigr)$. Again, we define $\mathcal{M}_{\xi, \rho}: C\bigl(U, V_\rho^{\oplus \dim(\rho)}\bigr) \to C\bigl(U, V_\rho^{\oplus \dim(\rho)}\bigr)$ in a similar fashion and denote $\tilde{\mathcal{L}}_\xi := \tilde{\mathcal{M}}_{\xi, 1}$ and $\mathcal{L}_\xi := \mathcal{M}_{\xi, 1}$. With this normalization, for all $a \in \mathbb R$, the maximal simple eigenvalue of $\mathcal{L}_a$ is $1$ with eigenvector $\frac{h_a}{h_0}$. Moreover, we have $\mathcal{L}_0^*(\nu_U) = \nu_U$.

We fix some related constants. By perturbation theory of operators as in \cite[Chapter 7]{Kat95} and \cite[Proposition 4.6]{PP90}, we can fix $a_0' > 0$ such that the map $[-a_0', a_0'] \to \mathbb R$ defined by $a \mapsto \lambda_a$ and the map $[-a_0', a_0'] \to C(\tilde{U}, \mathbb R)$ defined by $a \mapsto h_a$ are Lipschitz. We then fix $A_f > 0$ such that $\bigl|f_{(j, k)}^{(a)}(u) - f_{(j, k)}^{(0)}(u)\bigr| \leq A_f|a|$ for all admissible pairs $(j, k)$, $u \in \tilde{\mathtt{C}}[j, k]$, and $|a| \leq a_0'$. Fix $\overline{\tau} = \max_{(j, k)} \sup_{u \in \tilde{\mathtt{C}}[j, k]} \tau_{(j, k)}(u)$ and $\underline{\tau} = \min_{(j, k)} \inf_{u \in \tilde{\mathtt{C}}[j, k]} \tau_{(j, k)}(u)$. Fix
\begin{align*}
T_0 >{}&\max\bigg(\max_{(j, k)} \|\tau_{(j, k)}\|_{C^1}, \max_{(j, k)} \sup_{|a| \leq a_0'} \bigl\|f_{(j, k)}^{(a)}\bigr\|_{C^1}, \max_{(j, k)} \bigl\|\vartheta^{(j, k)}\bigr\|_{C^1}\bigg)
\end{align*}
which is possible by \cite[Lemma 4.1]{PS16}.

\subsection{Spectral bounds with holonomy}
\label{subsec:SpectralBoundsWithHolonomy}
We introduce some norms and seminorms. Let $\rho \in \widehat{M}$, and $H \in C\bigl(\tilde{U}, V_\rho^{\oplus \dim(\rho)}\bigr)$. Define $\|H\| \in C(\tilde{U}, \mathbb R)$ by $\|H\|(u) = \|H(u)\|_2$ for all $u \in \tilde{U}$, and if $\rho = 1$, we use the notation $|H| \in C(\tilde{U}, \mathbb R)$ instead. Define $\|H\|_\infty = \sup \|H\|$ and the $C^1$ seminorm and a modified $C^1$ norm by
\begin{align*}
|H|_{C^1} &= \sup_{u \in \tilde{U}}\|(dH)_u\|_{\mathrm{op}}, & \|H\|_{1, b} &= \|H\|_\infty + \frac{1}{\max(1, |b|)} |H|_{C^1}
\end{align*}
respectively. The usual $C^1$ norm is then $\|H\|_{C^1} = \|H\|_{1, 1}$. Define the Banach space
\begin{align*}
\mathcal{V}_\rho(\tilde{U}) = C^1\bigl(\tilde{U}, V_\rho^{\oplus \dim(\rho)}\bigr)
\end{align*}
of $C^1$ functions whose $C^1$ norm is \emph{bounded}. The following theorem gives spectral bounds of transfer operators with holonomy.

\begin{theorem}
\label{thm:TheoremFrameFlow}
There exist $\eta > 0$, $C > 0$, $a_0 > 0$, and $b_0 > 0$ such that for all $|a| < a_0$, $(b, \rho) \in \widehat{M}_0(b_0)$, $k \in \mathbb N$, and $H \in \mathcal{V}_\rho(\tilde{U})$, we have
\begin{align*}
\big\|\tilde{\mathcal{M}}_{\xi, \rho}^k(H)\big\|_2 \leq Ce^{-\eta k} \|H\|_{1, \|\rho_b\|}.
\end{align*}
\end{theorem}

\Cref{thm:TheoremExponentialMixingOfFrameFlow} is derived from \cref{thm:TheoremFrameFlow} using arguments of Pollicott and Paley--Wiener theory exactly as in \cite[\S\ 10]{SW21} and then using a convolutional argument as mentioned in the remark after \cite[Theorem 1.2]{SW21} (see also \cite[Theorem 3.1.4]{Sar22}). \Cref{thm:TheoremFrameFlow} follows from \cref{thm:FrameFlowDolgopyat} below, which captures the mechanism of Dolgopyat's method, by a standard inductive argument as in the proof after \cite[Theorem 5.4]{SW21}. To state \cref{thm:FrameFlowDolgopyat}, define the cone
\begin{align*}
K_B(\tilde{U}) &= \{h \in C^1(\tilde{U}, \mathbb R): h > 0, \|(dh)_u\|_{\mathrm{op}} \leq Bh(u) \text{ for all } u \in \tilde{U}\} \\
&= \{h \in C^1(\tilde{U}, \mathbb R): h > 0, |\log \circ h|_{C^1} \leq B\}.
\end{align*}

\begin{theorem}
\label{thm:FrameFlowDolgopyat}
There exist $m \in \mathbb N$, $\eta \in (0, 1)$, $E > \max\left(1, \frac{1}{b_0}, \frac{1}{\delta_\varrho}\right)$, $a_0 > 0$, $b_0 > 0$, and a set of operators $\{\mathcal{N}_{a, J}^H: C^1(\tilde{U}, \mathbb R) \to C^1(\tilde{U}, \mathbb R): H \in \mathcal{V}_\rho(\tilde{U}), |a| < a_0, J \in \mathcal{J}(b, \rho), \text{ for some } (b, \rho) \in \widehat{M}_0(b_0)\}$, where $\mathcal{J}(b, \rho)$ is some finite set for all $(b, \rho) \in \widehat{M}_0(b_0)$, such that
\begin{enumerate}
\item\label{itm:FrameFlowDolgopyatProperty1}	$\mathcal{N}_{a, J}^H\bigl(K_{E\|\rho_b\|}(\tilde{U})\bigr) \subset K_{E\|\rho_b\|}(\tilde{U})$ for all $H \in \mathcal{V}_\rho(\tilde{U})$, $|a| < a_0$, $J \in \mathcal{J}(b, \rho)$, and $(b, \rho) \in \widehat{M}_0(b_0)$;
\item\label{itm:FrameFlowDolgopyatProperty2}	$\left\|\mathcal{N}_{a, J}^H(h)\right\|_2 \leq \eta \|h\|_2$ for all $h \in K_{E\|\rho_b\|}(\tilde{U})$, $H \in \mathcal{V}_\rho(\tilde{U})$, $|a| < a_0$, $J \in \mathcal{J}(b, \rho)$, and $(b, \rho) \in \widehat{M}_0(b_0)$;
\item\label{itm:FrameFlowDolgopyatProperty3}	for all $|a| < a_0$ and $(b, \rho) \in \widehat{M}_0(b_0)$, if $H \in \mathcal{V}_\rho(\tilde{U})$ and $h \in K_{E\|\rho_b\|}(\tilde{U})$ satisfy
\begin{enumerate}[label=(1\alph*), ref=\theenumi(1\alph*)]
\item\label{itm:FrameFlowDominatedByh}	$\|H(u)\|_2 \leq h(u)$ for all $u \in \tilde{U}$,
\item\label{itm:FrameFlowLogLipschitzh}	$\|(dH)_u\|_{\mathrm{op}} \leq E\|\rho_b\|h(u)$ for all $u \in \tilde{U}$,
\end{enumerate}
then there exists $J \in \mathcal{J}(b, \rho)$ such that
\begin{enumerate}[label=(2\alph*), ref=\theenumi(2\alph*)]
\item\label{itm:FrameFlowDominatedByDolgopyat}	$\bigl\|\tilde{\mathcal{M}}_{\xi, \rho}^m(H)(u)\bigr\|_2 \leq \mathcal{N}_{a, J}^H(h)(u)$ for all $u \in \tilde{U}$,
\item\label{itm:FrameFlowLogLipschitzDolgopyat}	$\bigl\|\bigl(d\tilde{\mathcal{M}}_{\xi, \rho}^m(H)\bigr)_u\bigr\|_{\mathrm{op}} \leq E\|\rho_b\|\mathcal{N}_{a, J}^H(h)(u)$ for all $u, u' \in \tilde{U}$.
\end{enumerate}
\end{enumerate}
\end{theorem}

\section{Local non-integrability condition and non-concentration property}
\label{sec:LNIC&NCP}
This is the key section of the paper in which we prove the essential \emph{local non-integrability condition (LNIC)} and \emph{non-concentration property (NCP)}. Similar properties were proven for real hyperbolic manifolds in \cite{SW21}. For nonreal hyperbolic manifolds, one needs to deal with the subtle issue that the associated root system has an extra positive root. In this section we refine the techniques from \cite{SW21} to prove the appropriate generalized LNIC and NCP.

\subsection{Local non-integrability condition}
We start with a slight generalization of \cite[Definition 6.1]{SW21}. Choose unique isometric lifts $\tilde{\mathsf{R}}_j = \big[\tilde{\mathsf{U}}_j, \mathsf{S}_j\big] \subset \T^1(\mathbb{H}_\mathbb{K}^n)$ of $\tilde{R}_j$ for all $j \in \mathcal{A}$. For all $u \in \tilde{R}$, let $\tilde{u} \in \tilde{\mathsf{R}}$ denote the unique lift in $\tilde{\mathsf{R}}$. We then lift the section $F$ to $\mathsf{F}: \bigsqcup_{\gamma \in \Gamma} \gamma\tilde{\mathsf{R}} \to \Hol(\mathbb{H}_\mathbb{K}^n)$ in the natural way.

\begin{definition}[Associated sequence in $G$]
Let $z_1 \in \tilde{R}_1$ be the center. Consider some sequence of tangent vectors $(z_1, z_2, z_3, z_4, z_1) \in (\tilde{R}_1)^5$ such that $z_2 \in S_1$, $z_4 \in U_1$ and $z_3 = [z_4, z_2]$. Its lift to the universal cover is $(\tilde{z}_1, \tilde{z}_2, \tilde{z}_3, \tilde{z}_4, \tilde{z}_1) \in (\tilde{\mathsf{R}}_1)^5 \subset \T^1(\mathbb{H}_\mathbb{K}^n)^5 \cong (G/M)^5$. We define an \emph{associated sequence in $G$} to be the unique sequence $(g_1, g_2, \dotsc, g_5) \in G^5$ where
\begin{align*}
g_1 &= \mathsf{F}(\tilde{z}_1), \\
g_2 &= \mathsf{F}(\tilde{z}_2) \in g_1N^- \text{ such that } g_2M = \tilde{z}_2 \in \T^1(\mathbb{H}_\mathbb{K}^n) \cong G/M, \\
g_3 &\in g_2N^+ \text{ such that } g_3a_tM = \tilde{z}_3 \in \T^1(\mathbb{H}_\mathbb{K}^n) \cong G/M \text{ for some } t \in (-\underline{\tau}, \underline{\tau}), \\
g_4 &\in g_3N^- \text{ such that } g_4a_tM = \tilde{z}_4 \in \T^1(\mathbb{H}_\mathbb{K}^n) \cong G/M \text{ for some } t \in (-\underline{\tau}, \underline{\tau}), \\
g_5 &\in g_4N^+ \text{ such that } g_5a_tM = \tilde{z}_1 \in \T^1(\mathbb{H}_\mathbb{K}^n) \cong G/M \text{ for some } t \in (-\underline{\tau}, \underline{\tau}).
\end{align*}
\end{definition}

Continuing the same notation in the above definition, define
\begin{align*}
N_1^+ &= \{n^+ \in N^+: F(z_1)n^+ \in F(\tilde{U}_1)\} \subset N^+, \\
N_1^- &= \{n^- \in N^-: F(z_1)n^- \in F(S_1)\} \subset N^-
\end{align*}
where the first is open and the second is compact. Suppose that the sequence $(z_1, z_2, z_3, z_4, z_1)$ corresponds to some $n^+ \in N_1^+$ and $n^- \in N_1^-$ so that $F(z_4) = F(z_1)n^+$ and $F(z_2) = F(z_1)n^-$. Then, we define the map $\Xi: N_1^+ \times N_1^- \to AM$ by
\begin{align*}
\Xi(n^+, n^-) = g_5^{-1}g_1 \in AM.
\end{align*}
To view it as a function of the first coordinate for a fixed $n^- \in N_1^-$, we write $\Xi_{n^-}: N_1^+ \to AM$.

Let $z_1 \in \tilde{R}_1$ be the center. Let $j \in \mathbb N$ and $\alpha = (\alpha_0, \alpha_2, \dotsc, \alpha_{j - 1}, 1)$ be an admissible sequence. Then, there exists an element which we denote by $n_\alpha \in N_1^-$ such that
\begin{align*}
F(\mathcal{P}^j(\sigma^{-\alpha}(z_1))) = F(z_1)n_\alpha.
\end{align*}
This is well-defined because $\sigma^{-\alpha}(z_1) \in \mathtt{C}[\alpha] \subset U$.

In order to derive the LNIC in \cref{pro:FrameFlowLNIC}, we need the following lemmas regarding $\Xi$ whose proofs are as in \cite[Lemmas 6.2 and 6.3]{SW21}.

\begin{lemma}
\label{lem:BrinPesinInTermsOfHolonomy}
Let $j \in \mathbb N$, $\alpha = (\alpha_0, \alpha_1, \dotsc, \alpha_{j - 1}, 1)$ be an admissible sequence, and $n^- = n_\alpha \in N_1^-$. Let $u \in \tilde{U}_1$ and $n^+ \in N_1^+$ such that $F(u) = F(z_1)n^+$ where $z_1 \in \tilde{R}_1$ is the center. Then, we have
\begin{align*}
\Xi(n^+, n^-) = \Phi^\alpha(\sigma^{-\alpha}(z_1))^{-1}\Phi^\alpha(\sigma^{-\alpha}(u)).
\end{align*}
\end{lemma}

Define $\pi: \mathfrak{g} \to \LieA \oplus \LieM$ to be the projection with respect to the decomposition $\mathfrak{g} = \LieA \oplus \LieM \oplus \LieN^+ \oplus \LieN^-$. For all $\epsilon \in (0, \epsilon_0]$ where $\epsilon_0$ is as in \cref{subsec:MarkovSections}, define
\begin{align*}
N_{1, \epsilon}^+ &= \left\{n^+ \in N^+: F(z_1)n^+ \in F\left(W_\epsilon^{\mathrm{su}}(z_1)\right)\right\} \\
N_{1, \epsilon}^- &= \left\{n^- \in N^-: F(z_1)n^- \in F\left(W_\epsilon^{\mathrm{ss}}(z_1)\right)\right\}
\end{align*}
which are simply $\epsilon$-balls with respect to the induced metric $d_{N^+}$ and $d_{N^-}$ respectively, where $z_1 \in \tilde{R}_1$ is the center.

\begin{lemma}
\label{lem:BrinPesinDerivativeImageIsAdjointProjection}
For all $n^- \in N_1^-$, we have
\begin{align*}
(d\Xi_{n^-})_e = \pi \circ \Ad_{n^-}|_{\LieN^+} \circ (dh_{n^-})_e
\end{align*}
where $h_{n^-}: N_1^+ \to N^+$ is a diffeomorphism onto its image which is also smooth in $n^- \in N_{1, \epsilon_0}^-$ and satisfies $h_e = \Id_{N_1^+}$. Consequently, its image is $(d\Xi_{n^-})_e(\LieN^+) = \pi(\Ad_{n^-}(\LieN^+)) \subset \LieA \oplus \LieM$.
\end{lemma}

Helgason proved a very useful identity \cite[Chapter III, \S\ 1, Lemma 1.2]{Hel70} for noncompact connected semisimple Lie groups of arbitrary rank. We improve the argument to obtain the following stronger form of Helgason's identity in the rank one setting.

\begin{proposition}
\label{pro:BracketSpanLieA+LieM}
We have $[\LieSimple^-, \LieSimple^+] = \LieA \oplus \LieM$.
\end{proposition}

\begin{proof}
We use conventions from \cref{subsec:LieTheory}. Of course the root spaces satisfy $[\LieSimple^-, \LieSimple^+] \subset \LieA \oplus \LieM$ using the Jacobi identity; so we may focus on the reverse containment.

First we show that $\LieA \subset [\LieSimple^-, \LieSimple^+]$. Take any nonzero $s^- \in\LieSimple^-$. It suffices to show that $[s^-, s^+] \in \LieA \setminus \{0\}$ for $s^+ := \theta(s^-)$ since $\dim(\LieA) = 1$. Indeed, $s^+ \in \LieSimple^+$ because for all $t \in \LieA \cong \R$, we have
\begin{align*}
[t, s^+] = \theta[\theta(t), \theta(s^+)] = \theta[-t, s^-] = -t\theta(s^-) = -ts^+.
\end{align*}
Thus, $[s^-, s^+] \in [\LieSimple^-, \LieSimple^+] \subset \LieA \oplus \LieM$. Moreover, we have
\begin{align*}
\theta[s^-, s^+] = [\theta(s^-),\theta(s^+)] = [s^+, s^-] = -[s^-, s^+]
\end{align*}
which implies $[s^-, s^+] \in \LieP$. Hence, $[s^-, s^+] \in (\LieA\oplus\LieM) \cap \LieP = \LieA$. Recalling that $\langle \cdot, \cdot \rangle|_{\LieA \times \LieA} \asymp B|_{\LieA \times \LieA}$, for all $t \in \LieA \cong \R$, we have
\begin{align*}
\langle t, [s^-, s^+]\rangle \asymp B(t, [s^-, s^+]) = B([t, s^-], s^+) = tB(s^-, s^+) = -tB_\theta(s^-, s^-)
\end{align*}
which implies $[s^-, s^+] \asymp -B_\theta(s^-, s^-) \in \LieA \setminus \{0\}$ since $B_\theta$ is positive definite.

Next we show that $\LieM \subset [\LieSimple^-, \LieSimple^+]$. Suppose for the sake of contradiction that there exists a nonzero $m \in \LieM$ such that $B(m, [\LieSimple^-, \LieSimple^+]) = 0$. Let $\mathfrak{t} \subset \LieM$ be the maximal abelian subalgebra containing $m$ so that $\mathfrak{h} := \LieA \oplus \mathfrak{t} \subset \LieG$ is a Cartan subalgebra containing $m$. Consider the complexifications and root space decomposition $\mathfrak{h}^\C \subset \LieG^\C = \LieG_0^\C \oplus \bigoplus_{\beta \in \Phi^\C} \LieG_\beta^\C$ where the ordering of the root system $\Phi^\C$ is compatible with that of the restricted root system $\Phi$. Let $(\Phi^\C)^+$ be the set of positive roots and $(\Phi^\C)_\LieA^+$ be the set of positive roots which do not vanish on $\LieA$. Then \cite[Chapter VI, \S\ 3, Theorem 3.4]{Hel01} gives
\begin{align*}
\LieN^- = \LieG \cap \bigoplus_{\beta \in (\Phi^\C)_\LieA^+} \LieG_\beta^\C.
\end{align*}
Using the decomposition $\LieG = \LieA \oplus \LieM \oplus \LieN^+ \oplus \LieN^- = \LieA \oplus \LieM \oplus \theta(\LieN^-) \oplus \LieN^-$, we can derive that $[m, x] \neq 0$ for some nonzero $x \in \LieG_\beta^\C$ and $\beta \in (\Phi^\C)_\LieA^+$, because otherwise $m \in Z(\LieG) = \{0\}$ by semisimplicity. Let $\alpha = \beta|_{\LieA} \in \Phi^+$. Clearly, $\LieG_\beta^\C \subset \LieG_\alpha + i \LieG_\alpha$. So in fact, $[m, x] \neq 0$ for some nonzero $x \in \LieG_\alpha$. We can assume that $\alpha$ is the simple root because otherwise, writing $x = [x', x''] \in [\LieSimple^-, \LieSimple^-]$, we can conclude that $[m, x'] \neq 0$ or $[m, x''] \neq 0$ using the Jacobi identity. Now, assuming $\alpha$ is the simple root, we have $x \in \LieSimple^-$, $[m, x] \in \LieSimple^-$, and $\theta[m, x] \in \LieSimple^+$. Using our initial hypothesis, we calculate that
\begin{align*}
0 > -B_\theta([m, x], [m, x]) = B([m, x], \theta[m, x]) = B(m, [x, \theta[m, x]]) = 0
\end{align*}
which is a contradiction.
\end{proof}

\begin{remark}
The stronger form of Helgason's identity is essential precisely because the non-concentration property using $\LieN^+$ does not hold. When $\mathbb{K} = \R$, it coincides with the original form of Helgason's identity for which an alternative geometric proof is contained in \cite[Lemma 6.4]{SW21}.
\end{remark}

Although unnecessary for our purposes, it is of interest that \cref{pro:BracketSpanLieA+LieM} can be generalized to the higher rank setting as follows. Only for \cref{pro:HigherRankBracketSpanLieA+LieM}, allow $G$ to be a noncompact connected semisimple Lie group of arbitrary rank and denote by $\Pi \subset \Phi^+$ the set of simple roots.

\begin{proposition}
\label{pro:HigherRankBracketSpanLieA+LieM}
We have $\sum_{\alpha \in \Pi} [\LieG_\alpha, \LieG_{-\alpha}] = \LieA \oplus \LieM$.
\end{proposition}

The proof of \cref{pro:HigherRankBracketSpanLieA+LieM} is similar to the proof of \cref{pro:BracketSpanLieA+LieM} using the fact that $[\LieG_\alpha, \LieG_\beta] \subset \LieG_{\alpha + \beta}$ by the Jacobi identity can actually be upgraded to $[\LieG_\alpha, \LieG_\beta] = \LieG_{\alpha + \beta}$ by \cite[Proposition 3.1.13]{Sch17}, for all simple roots $\alpha, \beta \in \Pi$.

Returning to the rank one setting, we are now in a position to prove the following key lemma using \cref{pro:BracketSpanLieA+LieM}.

\begin{lemma}
\label{lem:am_ProjectionOfAdjointImage}
There exist $n_1^-, n_2^-, \dotsc, n_{j_{\mathrm{m}}}^- \in N_1^-$ for some $j_{\mathrm{m}} \in \mathbb N$ and $\delta > 0$ such that if $\eta_1^-, \eta_2^-, \dotsc, \eta_{j_{\mathrm{m}}}^- \in N_1^-$ with $d_{N^-}(\eta_j^-, n_j^-) \leq \delta$ for all $1 \leq j \leq j_{\mathrm{m}}$, then
\begin{align*}
\sum_{j = 1}^{j_{\mathrm{m}}} \pi\left(\Ad_{\eta_j^-}\circ (dh_{n_j^-})_e(\LieSimple^+)\right) = \LieA \oplus \LieM.
\end{align*}
\end{lemma}

\begin{proof}
Note that $[\LieN^-, \LieSimple^+] \subset \LieA \oplus \LieM \oplus \LieSimple^-$. Using the formula
\begin{align*}
\Ad_{e^{n^-}}(n^+) &= e^{\ad_{n^-}}(n^+) = \sum_{j = 0}^\infty \frac{1}{j!}(\ad_{n^-})^j(n^+) \\
&= n^+ + [n^-, n^+] + \frac{1}{2!}[n^-, [n^-, n^+]] + \frac{1}{3!}[n^-, [n^-, [n^-, n^+]]] + \dotsb,
\end{align*}
we have
\begin{align*}
\pi(\Ad_{e^{n^-}}(n^+)) = \pi[n^-,n^+] \qquad \text{for all $n^- \in \LieN^-$ and $n^+ \in \LieSimple^+$}.
\end{align*}
It follows from \cref{pro:BracketSpanLieA+LieM} that
\begin{align*}
\sum_{n^- \in N^-} \pi(\Ad_{n^-}(\LieSimple^+)) = \LieA \oplus \LieM.
\end{align*}

To prove the lemma, it suffices to show that
\begin{align*}
V := \sum_{n^- \in N_1^-} \pi(\Ad_{n^-}\circ (dh_{n^-})_e(\LieSimple^+)) = \LieA \oplus \LieM.
\end{align*}
Suppose for the sake of contradiction that $V \subset \LieA \oplus \LieM$ is a proper subspace. Let $L: \LieA \oplus \LieM \to \mathbb R$ be a linear map with $\ker(L) = V$. Fix $\hat{n}^- \in \LieN^-$ and $\hat{n}^+ \in \LieSimple^+$ such that $\pi(\Ad_{e^{\hat{n}^-}}(\hat{n}^+)) =\pi[\hat{n}^-,\hat{n}^+] \notin V$.

Without loss of generality, we may assume that $\mathsf{F}(\tilde{z}_1) = e \in G$. Let $O^- \subset N_{1, \epsilon_0}^-$ be an open neighborhood of $e \in N_{1, \epsilon_0}^-$. Let $H: O^- \to \LieN^+$ be a smooth map and $\tilde{L} : O^- \to \R$ be a nonzero smooth map defined by
\begin{align*}
\begin{split}
H(n^-) &= (dh_{n^-})_e(\hat{n}^+), \\
\tilde{L}(n^-) &= (L \circ \pi \circ \Ad_{n^-} \circ H)(n^-),
\end{split}
\qquad
\text{for all $n^- \in O^-$}.
\end{align*}
Note that $N_1^- \subset \tilde{L}^{-1}(0)$. Using the product rule, we calculate that
\begin{align*}
(d\tilde{L})_e(\tilde{n}^-) = L(\pi([\tilde{n}^-,\hat{n}^+] + (dH)_e(\tilde{n}^-))) = L(\pi[\tilde{n}^-,\hat{n}^+]) \qquad \text{for all $\tilde{n}^- \in \LieSimple^-$}.
\end{align*}
Thus, using the decomposition $\hat{n}^- = \hat{n}_{\LieSimple^-}^- + \hat{n}_{[\LieSimple^-, \LieSimple^-]}^- \in \LieSimple^- \oplus [\LieSimple^-, \LieSimple^-]$, we find that
\begin{align*}
d\tilde{L}_{e}(\hat{n}_{\LieSimple^-}^-) = L(\pi[\hat{n}_{\LieSimple^-}^-,\hat{n}^+]) = L(\pi[\hat{n}^-,\hat{n}^+]) \neq 0.
\end{align*}
Hence, using the implicit function theorem and shrinking $O^-$ if necessary, $\tilde{L}^{-1}(0) \subset N^-$ is a smooth submanifold of strictly smaller dimension. Since $N_1^- \cap O^- \subset \tilde{L}^{-1}(0)$, shrinking $O^-$ if necessary and using the diffeomorphism $N^- \to N^-e^-$, we conclude that $\limitset \cap O$ for some open subset $O \subset \partial_\infty\mathbb{H}_\mathbb{K}^n$ is contained in a smooth submanifold of $\partial_\infty\mathbb{H}_\mathbb{K}^n$ of strictly smaller dimension which is a contradiction by \cite[Proposition 3.12]{Win15} since $\Gamma < G$ is Zariski dense.
\end{proof}

With \cref{lem:am_ProjectionOfAdjointImage} in hand, the proof of \cref{pro:FrameFlowLNIC} is now similar to that of \cite[Proposition 6.5]{SW21}. Let $z_1 \in \tilde{R}_1$ be the center. Define the isometry $\Psi_0: N_{1, \epsilon_0}^+ \to W_{\epsilon_0}^{\mathrm{su}}(z_1)$ such that $F(\Psi_0(n^+)) = F(z_1)n^+$ for all $n^+ \in N_{1, \epsilon_0}^+$. For all $n^+ \in N_{1, \epsilon_0/2}^+$, also define the isometry $\iota(n^+): \tilde{U}_1 \to W_{\epsilon_0}^{\mathrm{su}}(z_1)$ by
\begin{align*}
\iota(n^+)(u) = \Psi_0(n^+\Psi_0^{-1}(u)).
\end{align*}
Define the diffeomorphism
\begin{align*}
\Psi = \Psi_0|_{N_1^+} \circ \exp: \log\bigl(N_{1, \epsilon_0}^+\bigr) \to W_{\epsilon_0}^{\mathrm{su}}(z_1).
\end{align*}
Denote $\check{u} = \Psi^{-1}(u)$ for all $u \in \tilde{U}_1$. 

\begin{proposition}[LNIC]
\label{pro:FrameFlowLNIC}
There exist $\varepsilon_2 \in (0, 1)$, $m_0 \in \mathbb N$, $j_{\mathrm{m}} \in \mathbb N$, and an open neighborhood $\mathcal{U} \subset \tilde{U}_1$ of the center $z_1 \in \tilde{R}_1$ with $\mathcal{U} \cap \Omega \subset U_1$ such that for all $m \geq m_0$, there exist sections $v_j = \sigma^{-\alpha_j}: \tilde{U}_1 \to \tilde{U}_{\alpha_{j, 0}}$ for some admissible sequences $\alpha_j = (\alpha_{j, 0}, \alpha_{j, 1}, \dotsc, \alpha_{j, m - 1}, 1)$ for all integers $0 \leq j \leq j_{\mathrm{m}}$ such that for all $u \in \mathcal{U}$ and $\omega \in \LieA \oplus \LieM$ with $\|\omega\| = 1$, there exist $1 \leq j \leq j_{\mathrm{m}}$ and $Z := (d\Psi)_{\check{u}}(Z') \in (d\Psi)_{\check{u}}(\LieSimple^+)$ with $\|Z'\| = 1$ such that
\begin{align*}
|\langle (d\BP_{j, u})_u(Z), \omega \rangle| \geq \varepsilon_2
\end{align*}
where we define $\BP_j: \tilde{U}_1 \times \tilde{U}_1 \to AM$ by
\begin{align*}
\BP_j(u, u') = \Phi^{\alpha_0}(v_0(u))^{-1} \Phi^{\alpha_0}(v_0(u')) \Phi^{\alpha_j}(v_j(u'))^{-1} \Phi^{\alpha_j}(v_j(u))
\end{align*}
and we denote $\BP_{j, u} = \BP_j(u, \cdot)$ for all $u, u' \in \tilde{U}_1$ and $1 \leq j \leq j_{\mathrm{m}}$. Moreover, $v_0(\mathcal{U}), v_1(\mathcal{U}), \dotsc, v_{j_{\mathrm{m}}}(\mathcal{U})$ are mutually disjoint.
\end{proposition}

We refer to \cite[Proposition 6.5]{SW21} for the details of the proof and prefer to focus on the idea of the proof and emphasize the differences. We use the same notation as in \cref{pro:FrameFlowLNIC}. Firstly, for all $1 \leq j \leq j_{\mathrm{m}}$, we can relate $\BP_j$ to $\Xi$ via the following observation obtained using definitions and \cref{lem:BrinPesinInTermsOfHolonomy}:
\begin{align}
\label{eqn:BP_and_Xi_Relation}
\BP_j(u, u') = \Xi_{n_{\alpha_0}}(\varphi(u))^{-1} \Xi_{n_{\alpha_0}}(\varphi(u')) \Xi_{n_{\alpha_j}}(\varphi(u'))^{-1} \Xi_{n_{\alpha_j}}(\varphi(u))
\end{align}
for all $u, u' \in \tilde{U}_1$, where $\varphi: \tilde{U}_1 \to N_1^+$ is a diffeomorphism such that
\begin{align*}
F(u) = F(z_1) \varphi(u) \qquad \text{for all $u \in \tilde{U}_1$}.
\end{align*}
Taking $u = z_1$ in \cref{eqn:BP_and_Xi_Relation}, the first and the last factors are $e$. Taking the differential at $u' = z_1$ and evaluating at any $Z := (d\Psi)_0(Z') \in (d\Psi)_0(\LieSimple^+)$ with $\|Z'\| = 1$ (note that $\check{z}_1 = 0$) gives
\begin{align*}
d\bigl(\BP_j(z_1, \cdot)\bigr)_{z_1}(Z) = \bigl(d\Xi_{n_{\alpha_0}}\bigr)_e((d\varphi)_{z_1}(Z)) - \bigl(d\Xi_{n_{\alpha_j}}\bigr)_e((d\varphi)_{z_1}(Z)).
\end{align*}
Observe that we \emph{do not} take $Z' \in \LieN^+$. Write $W = (d\varphi)_{z_1}(Z)$. Applying \cref{lem:BrinPesinDerivativeImageIsAdjointProjection}, taking inner products with any $\omega \in \LieA \oplus \LieM$ with $\|\omega\| = 1$, and using the reverse triangle inequality gives
\begin{align*}
\Bigl|\Bigl\langle d\bigl(\BP_j(z_1, \cdot)\bigr)_{z_1}(Z), \omega\Bigr\rangle\Bigr| \geq{}&\bigl|\bigl\langle\pi\bigl(\Ad_{n_{\alpha_j}}\bigl(\bigl(dh_{n_{\alpha_j}}\bigr)_e(W)\bigr)\bigr), \omega\bigr\rangle\bigr| \\
{}&- \bigl|\bigl\langle \pi\bigl(\Ad_{n_{\alpha_0}}\bigl(\bigl(dh_{n_{\alpha_0}}\bigr)_e(W)\bigr)\bigr), \omega\bigr\rangle\bigr|.
\end{align*}
We then construct sequences $\alpha_j$ and $\alpha_0$, and find a vector $W$, using \cref{lem:am_ProjectionOfAdjointImage} and the topological mixing property of the geodesic flow, so that the first term is large and the second term is small to bound the difference from below by a positive constant. We can actually ensure that $n_{\alpha_0}$ is close to $e \in N_1^-$ and $n_{\alpha_j}$ is close to $n_j^- \in N_1^-$ from \cref{lem:am_ProjectionOfAdjointImage}. The crucial point to observe here is that this step works because, in contrast to \cite[Lemma 6.4]{SW21}, we are able prove \cref{lem:am_ProjectionOfAdjointImage} with $\LieSimple^+$ instead of $\LieN^+$. Finally, we extend the bound using smoothness of $\BP_j$ and compactness of the unit ball centered at $0$ in $\LieA \oplus \LieM$.

\subsection{Non-concentration property}
We first introduce several notations. Without loss of generality, we assume that $z_1 = \Gamma eM \in \Gamma \backslash G/M$ so that $\tilde{z}_1^\pm = e^\pm$ and $\limitset \setminus \{e^-\} \subset N^+e^+$. Abusing notation, we denote $\exp: \mathfrak{n}^+ \to \partial_\infty\mathbb{H}_\mathbb{K}^n \cong G/P$. Define $\loglimitset = \log(\limitset \cap N^+e^+)$. We also use the notation $n_x^+ = \exp(x) \in N^+$ for all $x \in \LieN^+$. The left translation action of $N^+$ on itself induces an action on $\LieN^+$ for which we simply write $n_x^+ \cdot y = \log(n_x^+\exp(y))$ for all $x, y \in \LieN^+$. Now, $N^+$ has a left $N^+$-invariant Carnot--Carath\'{e}odory metric (see for example \cite[\S\ 3]{Bou18}). Pulling back via $\exp$, we instead endow $\LieN^+$ with the Carnot--Carath\'{e}odory metric which we denote by $d_{\mathrm{CC}}$. We denote by $B_\epsilon^{\mathrm{CC}}(x) \subset \LieN^+$ the open Carnot--Carath\'{e}odory ball of radius $\epsilon > 0$ centered at $x \in \LieN^+$ which has the properties
\begin{align*}
n_y^+ \cdot B_\epsilon^{\mathrm{CC}}(x) &= B_\epsilon^{\mathrm{CC}}(n_y^+ \cdot x), & \Ad_{a_t}(B_\epsilon^{\mathrm{CC}}(0)) &= B_{e^{-t}\epsilon}^{\mathrm{CC}}(0),
\end{align*}
for all $y \in \LieN^+$ and $t \in \R$. Define the convenient inner product $\langle \cdot, \cdot\rangle_\perp$ on $\LieN^+ = \LieSimple^+ \oplus [\LieSimple^+, \LieSimple^+]$ such that the decomposition is orthogonal and the restrictions to $\LieSimple^+$ and $[\LieSimple^+, \LieSimple^+]$ both coincide with $\langle\cdot, \cdot\rangle$. Denote by $\|\cdot\|_\perp$ the corresponding norm. For all $\varkappa \in [0, 1]$, denote
\begin{align}
\label{eqn:RingLikeSubset}
\mathfrak{R}_\varkappa(\LieSimple^+) &= \Bigl\{w \in \LieN^+: \|w\|_\perp = \|w\| = 1, \sup_{w' \in \LieSimple^+, \|w'\| = 1} |\langle w, w'\rangle_\perp| \geq \varkappa\Bigr\}.
\end{align}
Then $\{w \in \LieSimple^+: \|w\| = 1\} \subset \mathfrak{R}_\varkappa(\LieSimple^+) \subset \{w \in \LieN^+: \|w\| = 1\}$ and the two containments are equalities when $\varkappa = 1$ and $\varkappa = 0$ respectively. The following proposition is the required NCP.

\begin{proposition}[NCP]
\label{pro:NonConcentrationProperty}
Let $\varkappa \in (0, 1]$. There exists $\delta \in (0, 1)$ such that for all $x \in \loglimitset \cap \overline{B_1^{\mathrm{CC}}(0)}$, $\epsilon \in (0, 1)$, and $w \in \mathfrak{R}_\varkappa(\LieSimple^+)$, there exists $y \in \loglimitset \cap \bigl(B_\epsilon^{\mathrm{CC}}(x) \setminus B_{\epsilon\delta}^{\mathrm{CC}}(x)\bigr)$ such that $|\langle (n_x^+)^{-1} \cdot y, w \rangle| \geq \epsilon \delta$.
\end{proposition}

\begin{proof}
Let $\varkappa \in (0, 1]$. Observe that using compactness of $\overline{B_1^{\mathrm{CC}}(0)}$ and \cite[Proposition 3.12]{Win15} since $\Gamma < G$ is Zariski dense, the proposition easily holds with the restriction $\epsilon \in [\epsilon', 1)$ for some $\epsilon' > 0$. Now, suppose for the sake of contradiction that the proposition is false. Then for all $j \in \mathbb N$, taking $\delta_j = \frac{1}{j}$, there exist $x_j \in \loglimitset \cap \overline{B_1^{\mathrm{CC}}(0)}$, $\epsilon_j \in (0, 1)$, and $w_j \in \mathfrak{R}_\varkappa(\LieSimple^+)$ such that $|\langle (n_{x_j}^+)^{-1} \cdot y, w_j \rangle| \leq \frac{\epsilon_j}{j}$ for all $y \in \loglimitset \cap \bigl(B_{\epsilon_j}^{\mathrm{CC}}(x_j) \setminus B_{\epsilon_j/j}^{\mathrm{CC}}(x_j)\bigr)$. This can be written as
\begin{align}
\label{eqn:IfLemmaIsFalse}
\limitset \cap \exp\bigl(B_{\epsilon_j}^{\mathrm{CC}}(x_j) \setminus B_{\epsilon_j/j}^{\mathrm{CC}}(x_j)\bigr) \subset \exp\left\{y \in \LieN^+: |\langle (n_{x_j}^+)^{-1} \cdot y, w_j \rangle| \leq \frac{\epsilon_j}{j}\right\}
\end{align}
for all $j \in \mathbb N$. Based on the initial observation, without loss of generality, we can pass to subsequences so that $\lim_{j \to \infty} \epsilon_j = 0$.

We will use the self-similarity of the fractal set $\limitset$. For all $j \in \mathbb N$, we have $n_{x_j}^+ \in B_1(e) \subset G$ with $\bigl(n_{x_j}^+\bigr)^+ = \exp(x_j) \in \limitset$ and $\bigl(n_{x_j}^+\bigr)^- = e^- \in \limitset$. By convex cocompactness, there exists a compact subset $\Omega_0 \subset G$ such that $n_{x_j}^+ a_t \in \Gamma \Omega_0$ for all $t \in \mathbb R$ and $j \in \mathbb N$. Hence, for all $j \in \mathbb N$, setting $t_j = -\log(\epsilon_j)$, there exist $\gamma_j \in \Gamma$ and $g_j \in \Omega_0$ such that $n_{x_j}^+ a_{t_j} = \gamma_j g_j$. Then, for all $j \in \mathbb N$, we have $g_j a_{-t_j} (n_{x_j}^+)^{-1} = \gamma_j^{-1}$ which preserves $\limitset$. This captures the self-similarity of $\limitset$.

By compactness, we can pass to subsequences so that $\lim_{j \to \infty} g_j = g \in \Omega_0$ and $\lim_{j \to \infty} w_j = w \in \mathfrak{R}_\varkappa(\LieSimple^+)$. Now, for all $j \in \mathbb N$, applying $g_j a_{-t_j} (n_{x_j}^+)^{-1}$ in \cref{eqn:IfLemmaIsFalse} and recalling that the right $A$-action on $\partial_\infty\mathbb{H}_\mathbb{K}^n \cong G/P$ is trivial, we have
\begin{align*}
&\limitset \cap g_j \cdot \exp\bigl(B_1^{\mathrm{CC}}(0) \setminus B_{1/j}^{\mathrm{CC}}(0)\bigr) \\
\subset{}&g_ja_{-t_j} \cdot \exp\left((n_{x_j}^+)^{-1} \cdot \left\{y \in \LieN^+: |\langle (n_{x_j}^+)^{-1} \cdot y, w_j \rangle| \leq \frac{\epsilon_j}{j}\right\}\right) \cdot a_{t_j} \\
={}&g_j \cdot \exp \left(\Ad_{a_{-t_j}} \left\{y \in \LieN^+: |\langle y, w_j \rangle| \leq \frac{\epsilon_j}{j}\right\}\right) \\
={}&g_j \cdot \exp\left\{y \in \LieN^+: \bigl|\bigl\langle \Ad_{a_{t_j}}(y), w_j \bigr\rangle\bigr| \leq \frac{\epsilon_j}{j}\right\} \\
={}&g_j \cdot \exp\left\{y_{\LieSimple^+} + y_{[\LieSimple^+, \LieSimple^+]} \in \LieSimple^+ \oplus [\LieSimple^+, \LieSimple^+]: \left|\left\langle e^{-t_j}y_{\LieSimple^+} + e^{-2t_j}y_{[\LieSimple^+, \LieSimple^+]}, w_j \right\rangle\right| \leq \frac{\epsilon_j}{j}\right\} \\
={}&g_j \cdot \exp\left\{y_{\LieSimple^+} + y_{[\LieSimple^+, \LieSimple^+]} \in \LieSimple^+ \oplus [\LieSimple^+, \LieSimple^+]: \left|\left\langle y_{\LieSimple^+} + e^{-t_j}y_{[\LieSimple^+, \LieSimple^+]}, w_j \right\rangle\right| \leq \frac{1}{j}\right\}.
\end{align*}
Then in the limit $j \to \infty$, we have
\begin{align}
\label{eqn:SmallPieceOfLimitSetContainedInCodimensionOneSubmanifold}
\limitset \cap g \cdot \exp\bigl(B_1^{\mathrm{CC}}(0)\bigr) \subset g \cdot \exp\left\{y \in \LieN^+: \langle y, L^*(w)\rangle = 0\right\}
\end{align}
where $L: \LieN^+ \to \LieN^+$ is the orthogonal projection map onto $\LieSimple^+$ with respect to $\langle \cdot, \cdot\rangle_\perp$. But $L^*(w) \neq 0$ because $\mathfrak{R}_\varkappa(\LieSimple^+) \ni w$ does not intersect $(\LieSimple^+)^\perp = \im(L)^\perp = \ker(L^*)$ with respect to $\langle \cdot, \cdot\rangle_\perp$ as $\varkappa > 0$. Thus \cref{eqn:SmallPieceOfLimitSetContainedInCodimensionOneSubmanifold} contradicts \cite[Proposition 3.12]{Win15} since $\Gamma < G$ is Zariski dense.
\end{proof}

Due to \cite[Eq. (6)]{SW21}, taking the adjoint with respect to $\langle\cdot, \cdot\rangle_\perp$ on $\T_{\check{y}}(\LieN^+) \cong \LieN^+$, we can fix the positive constant
\begin{align}
\label{eqn:ConstantBPPsiAdjoint}
C_{\BP, \Psi}^* &= \sup_{x, y \in \tilde{U}_1, j \in \{1, 2, \dotsc, j_{\mathrm{m}}\}} \|(d\BP_{j, x} \circ \Psi)_{\check{y}}^*\|_{\mathrm{op}}
\end{align}
independent of the sections and their length provided by \cref{pro:FrameFlowLNIC}. We then fix the positive constant
\begin{align}
\label{eqn:ConstantVarkappa0}
\varkappa_0 &= \frac{\varepsilon_2}{C_{\BP, \Psi}^*}.
\end{align}
Finally, fix $\varepsilon_3 \in (0, 1)$ to be the $\delta$ provided by \cref{pro:NonConcentrationProperty} corresponding to $\varkappa_0$.

\section{Construction of Dolgopyat operators}
\label{sec:ConstructionOfDolgopyatOperators}
In this section, we will construct the Dolgopyat operators. We start with some preliminary lemmas and fix many required constants.

\Cref{lem:ComparingExpWithBP} applies for any arbitrary set of sections of any length provided by \cref{pro:FrameFlowLNIC}. Shrinking $\mathcal{U} \subset \tilde{U}_1$ if necessary, we can assume henceforth that it is a convex open subset and $\Psi^{-1}(\mathcal{U}) \subset \overline{B_1^{\mathrm{CC}}(0)}$. Due to \cite[Eq. (6)]{SW21}, we can fix the positive constant
\begin{align*}
C_{\BP} = \sup_{x, y \in \tilde{U}_1, j \in \{1, 2, \dotsc, j_{\mathrm{m}}\}} \|(d\BP_{j, x})_y\|_{\mathrm{op}}
\end{align*}
independent of the sections and their length provided by \cref{pro:FrameFlowLNIC}. There exists $\delta_0 > 0$ such that any pair of points in $B_{\delta_0}^{AM}(e) \subset AM$ has a unique geodesic through them. For all $x \in \mathcal{U}$ and $z \in \T_x(\mathcal{U})$ with $\|z\| \leq d(x, \partial(\mathcal{U}))$, denote by $\gamma_z: [0, 1] \to \mathcal{U}$ the geodesic with $\gamma_z(0) = x$ and $\gamma_z'(0) = z$. Let $x, y \in \mathcal{U}$ and $z \in \T_x(\mathcal{U})$ such that $\gamma_z(1) = y$. Let $1 \leq j \leq j_{\mathrm{m}}$ be an integer. \Cref{lem:ComparingExpWithBP} can be proved as in \cite[Lemma 7.1]{SW21} using the curve
\begin{align*}
\varphi^{\mathrm{BP}}_{j, x, z} = \BP_{j, x} \circ \gamma_z: [0, 1] \to AM
\end{align*}
which has endpoints $\varphi^{\mathrm{BP}}_{j, x, z}(0) = e$ and $\varphi^{\mathrm{BP}}_{j, x, z}(1) = \BP_{j, x}(y) = \BP_j(x, y)$.

\begin{lemma}
\label{lem:ComparingExpWithBP}
There exists $C_{\exp, \BP} > 0$ such that for all $1 \leq j \leq j_{\mathrm{m}}$ and $x, y \in \mathcal{U}$ with $d(x, y) < \frac{\delta_0}{C_{\BP}}$, we have
\begin{align*}
d_{AM}\left(\exp(Z), \BP_j(x, y)\right) \leq C_{\exp, \BP} d(x, y)^2
\end{align*}
where $z \in \T_x(\mathcal{U})$ such that $\gamma_z(1) = y$ and $Z = (d\BP_{j, x})_x(z)$.
\end{lemma}

The following \cref{lem:SigmaHyperbolicity} is a basic bound derived from the hyperbolicity of the geodesic flow, which we will freely use in the rest of the paper without further comments.

\begin{lemma}
\label{lem:SigmaHyperbolicity}
There exist $c_0 \in (0, 1)$ and $\kappa_1 > \kappa_2 > 1$ such that for all $j \in \mathbb N$ and admissible sequences $\alpha = (\alpha_0, \alpha_1, \dotsc, \alpha_j)$, we have
\begin{align*}
\frac{c_0}{\kappa_1^j} \leq \|(d\sigma^{-\alpha})_u\|_{\mathrm{op}} \leq \frac{1}{c_0\kappa_2^j} \qquad \text{for all $u \in \tilde{U}_{\alpha_j}$}.
\end{align*}
\end{lemma}

Recall that Dolgopyat's method works successfully when the derivative of $\rho_b$ is large. This motivated the definition of $\widehat{M}_0(b_0)$ for all $b_0 > 0$. This criteria appears for the Lasota--Yorke \cite{LY73} type inequalities in \cref{lem:FrameFlowPreliminaryLogLipschitz}.

\begin{lemma}
\label{lem:FrameFlowPreliminaryLogLipschitz}
There exists $A_0 > 0$ such that for all $\xi \in \mathbb C$ with $|a| < a_0'$, if $(b, \rho) \in \widehat{M}_0(1)$, then for all $k \in \mathbb N$, we have
\begin{enumerate}
\item\label{itm:FrameFlowPreliminaryLogLipschitzProperty1}	if $h \in K_B(\tilde{U})$ for some $B > 0$, then we have $\tilde{\mathcal{L}}_a^k(h) \in K_{B'}(\tilde{U})$ where $B' = A_0\left(\frac{B}{\kappa_2^k} + 1\right)$;
\item\label{itm:FrameFlowPreliminaryLogLipschitzProperty2}	if $H \in \mathcal{V}_\rho(\tilde{U})$ and $h \in C^1(\tilde{U}, \mathbb R)$ satisfy $\|(dH)_u\|_{\mathrm{op}} \leq Bh(u)$ for all $u \in \tilde{U}$, for some $B > 0$, then we have
\begin{align*}
\bigl\|\bigl(d\tilde{\mathcal{M}}_{\xi, \rho}^k(H)\bigr)_u\bigr\|_{\mathrm{op}} \leq A_0\left(\frac{B}{\kappa_2^k}\tilde{\mathcal{L}}_a^k(h)(u) + \|\rho_b\|\tilde{\mathcal{L}}_a^k\|H\|(u)\right) \qquad \text{for all $u \in \tilde{U}$}.
\end{align*}
\end{enumerate}
\end{lemma}

Fix $A_0 > 0$ provided by \cref{lem:FrameFlowPreliminaryLogLipschitz}. Fix $m_1 \in \mathbb N$ sufficiently large and for all $k \in \mathcal{A}$, fix a cylinder $\mathtt{C}_k \subset U_1$ with $\len(\mathtt{C}_k) = m_1$ such that $\overline{\mathtt{C}_k} \subset \mathcal{U}$ and $\sigma^{m_1}(\mathtt{C}_k) = \interior(U_k)$. Fix $C_{\mathrm{Vit}} = \min_{k \in \mathcal{A}}d(\overline{\mathtt{C}_k}, \partial(\mathcal{U}))$. Let the corresponding sections be $\mathtt{v}_k: \tilde{U}_k \to \tilde{U}_1$ for all $k \in \mathcal{A}$. Let $z_1 \in \tilde{R}_1$ be the center and endow $W_{\epsilon_0}^{\mathrm{su}}(z_1)$ with the Carnot--Carath\'{e}odory metric using the map $\Psi$ which is then an isometry. We denote by $W_\epsilon^{\mathrm{su}, \mathrm{CC}}(u) \subset W_{\epsilon_0}^{\mathrm{su}}(z_1)$ the open Carnot--Carath\'{e}odory ball of radius $\epsilon \in (0, \epsilon_{\mathrm{CC}})$ centered at $u \in \tilde{U}_1$ where $\epsilon_{\mathrm{CC}} > 0$ is some upper bound so that the containment holds. Now, fix positive constants
\begin{align}
\label{eqn:Constant_b_0}
b_0 &= 1, \\
\label{eqn:ConstantE}
E &> \frac{2A_0}{\delta_{1, \varrho}}, \\
\delta_1 &< \frac{\varepsilon_1\varepsilon_2\varepsilon_3\delta_\Psi }{14}, \\
\label{eqn:Constantepsilon1}
\epsilon_1 &< \min\left(C_{\mathrm{Vit}}, \epsilon_{\mathrm{CC}}, \frac{2\delta_0 \delta_{1, \varrho}}{C_{\BP}}, \frac{4\delta_1}{C_{\BP}^2}, \frac{4\delta_1 \delta_{1, \varrho}}{C_{\exp, \BP}}, \frac{1}{\delta_1}, \frac{c_0C_{\mathrm{Ano}}C_\phi\hat{\delta}e^{\hat{\delta}}}{5\kappa_1^{m_1}\delta_{1, \varrho}}\right), \\
\label{eqn:Constantepsilon2}
\epsilon_2 &< \min\left(\frac{\varepsilon_3\epsilon_1}{4N}, \frac{\delta_1\epsilon_1}{4N(A_0 + \delta_1)}\right), \\
\label{eqn:Constantepsilon3}
\epsilon_3 &= \frac{c_0\kappa_2^{m_1}\epsilon_2}{2}, \\
\label{eqn:Constantepsilon4}
\epsilon_4 &= 10c_0^{-1}\kappa_1^{m_1}\epsilon_1, \\
\label{eqn:Constantm2}
m_2 &> m_0 \text{ such that } \kappa_2^{m_2} > \max\left(8A_0, \frac{4EN\epsilon_2}{c_0\log(2)}, \frac{32EN\epsilon_2}{c_0}, \frac{4E}{c_0\delta_1}\right), \\
\label{eqn:Constantmu}
\mu &< \min\left(\frac{E\epsilon_2}{2N}, \frac{1}{4N}, \frac{1}{16 \cdot 16e^{2m_2 T_0} \cdot N}\arccos\left(1 - \frac{(\delta_1 \epsilon_1)^2}{2}\right)^2\right).
\end{align}
Here, $C_\phi$ is defined in \cite[Subsection 9.2]{SW21} and only depends on the Markov section which was fixed in \cref{subsec:MarkovSections}. Fix $m = m_1 + m_2$. Fix admissible sequences $\alpha_j = (\alpha_{j, 0}, \alpha_{j, 1}, \dotsc, \alpha_{j, m_2 - 1}, 1)$ and corresponding sections $v_j = \sigma^{-\alpha_j}: \tilde{U}_1 \to \tilde{U}_{\alpha_{j, 0}}$ provided by \cref{pro:FrameFlowLNIC} for all $0 \leq j \leq j_{\mathrm{m}}$. Fix corresponding maps $\BP_j: \tilde{U}_1 \times \tilde{U}_1 \to AM$ provided by \cref{pro:FrameFlowLNIC} for all $1 \leq j \leq j_{\mathrm{m}}$.

Let $(b, \rho) \in \widehat{M}_0(b_0)$ and $k \in \mathcal{A}$. Using $\Psi$ and the Vitali covering lemma on $\LieN^+ \cong \R^{\dim(\LieN^+)} = \R^{\dim_\R(\mathbb K)n - 1}$, we can choose a finite subset $\big\{x_{k, r, 1}^{(b, \rho)} \in \mathtt{C}_k: r \in \big\{1, 2, \dotsc, r_k^{(b, \rho)}\big\}\big\} \subset \mathtt{C}_k$ for some $r_k^{(b, \rho)} \in \mathbb N$ and corresponding open balls $C_{k, r}^{(b, \rho)} = W_{\epsilon_1/\|\rho_b\|}^{\mathrm{su}, \mathrm{CC}}\big(x_{k, r, 1}^{(b, \rho)}\big) \subset \mathcal{U}$ and $\hat{C}_{k, r}^{(b, \rho)} = W_{5\epsilon_1/\|\rho_b\|}^{\mathrm{su}, \mathrm{CC}}\big(x_{k, r, 1}^{(b, \rho)}\big)$ for all $1 \leq r \leq r_k^{(b, \rho)}$ such that $C_{k, r}^{(b, \rho)} \cap C_{k, r'}^{(b, \rho)} = \varnothing$ for all $1 \leq r, r' \leq r_k^{(b, \rho)}$ with $r \neq r'$ and $\mathtt{C}_k \subset \bigcup_{r = 1}^{r_k^{(b, \rho)}} \hat{C}_{k, r}^{(b, \rho)}$. Recall the notation $\check{x}_{k, r, 1}^{(b, \rho)} = \Psi^{-1}\big(x_{k, r, 1}^{(b, \rho)}\big)$ for all $1 \leq r \leq r_k^{(b, \rho)}$.

\begin{lemma}
\label{lem:PartnerPointInZariskiDenseLimitSetForBPBound}
For all $(b, \rho) \in \widehat{M}_0(b_0)$, $x_1 \in \mathcal{U} \cap \Omega$, and $\omega \in V_\rho^{\oplus \dim(\rho)}$ with $\|\omega\|_2 = 1$, there exist $1 \leq j \leq j_{\mathrm{m}}$ and $\check{x}_2 \in \limitset \cap \left(B_{s_1}^{\mathrm{CC}}(\check{x}_1) \setminus B_{s_2}^{\mathrm{CC}}(\check{x}_1)\right)$ such that
\begin{align*}
\left\|d\rho_b\left(d(\BP_{j, x_1} \circ \iota(n_{\check{x}_1}^+) \circ \Psi)_0(z)\right)(\omega)\right\|_2 \geq 7\delta_1\epsilon_1
\end{align*}
where $s_1 = \frac{\epsilon_1}{2\|\rho_b\|}$, $s_2 = \frac{\varepsilon_3\epsilon_1}{2\|\rho_b\|}$, and $z = \bigl(0, \bigl(n_{\check{x}_1}^+\bigr)^{-1} \cdot \check{x}_2\bigr) \in \T_0(\LieN^+)$.
\end{lemma}

\begin{proof}
Let $(b, \rho) \in \widehat{M}_0(b_0)$, $x_1 \in \mathcal{U} \cap \Omega$, and $\omega \in V_\rho^{\oplus \dim(\rho)}$ with $\|\omega\|_2 = 1$. Fix $s_1 = \frac{\epsilon_1}{2\|\rho_b\|}$, $s_2 = \frac{\varepsilon_3\epsilon_1}{2\|\rho_b\|}$. Define the linear maps $L_1 = d(\iota(n_{\check{x}_1}^+) \circ \Psi)_0: \T_0(\LieN^+) \to \T_{x_1}(\mathcal{U})$, $L_{2, j} = (d\BP_{j, x_1})_{x_1}: \T_{x_1}(\mathcal{U}) \to \LieA \oplus \LieM$, and $L_3: \LieA \oplus \LieM \to V_\rho^{\oplus \dim(\rho)}$ by $L_3(w) = d\rho_b(w)(\omega)$ for all $w \in \LieA \oplus \LieM$. It suffices to find $1 \leq j \leq j_{\mathrm{m}}$ and $\check{x}_2 \in \limitset \cap \left(B_{s_1}^{\mathrm{CC}}(\check{x}_1) \setminus B_{s_2}^{\mathrm{CC}}(\check{x}_1)\right)$ such that
\begin{align*}
|\langle (L_3 \circ L_{2, j} \circ L_1)(z), \omega_0 \rangle| = |\langle z, (L_1^* \circ L_{2, j}^* \circ L_3^*)(\omega_0) \rangle| \geq 7\delta_1\epsilon_1
\end{align*}
for some $\omega_0 \in V_\rho^{\oplus \dim(\rho)}$ with $\|\omega_0\|_2 = 1$, where $z = \bigl(0, \bigl(n_{\check{x}_1}^+\bigr)^{-1} \cdot \check{x}_2\bigr) \in \T_0(\LieN^+)$, and the adjoints are taken with respect to $\langle\cdot, \cdot\rangle_\perp$ on $\T_0(\LieN^+) \cong \LieN^+$. By \cref{lem:maActionLowerBound}, $\|L_3\|_{\mathrm{op}} \geq \varepsilon_1\|\rho_b\|$ which implies $\|L_3^*\|_{\mathrm{op}} \geq \varepsilon_1\|\rho_b\|$. Hence there exists $\omega_0 \in V_\rho^{\oplus \dim(\rho)}$ with $\|\omega_0\|_2 = 1$ such that $\|L_3^*(\omega_0)\| \geq \varepsilon_1\|\rho_b\|$. \Cref{pro:FrameFlowLNIC} implies that there exists $1 \leq j \leq j_{\mathrm{m}}$ such that $\|L_{2, j}^*(L_3^*(\omega_0))\| \geq \varepsilon_2 \varepsilon_1\|\rho_b\|$. Using a previous bound for $\Psi$ and the fact that $\iota(n_{\check{x}_1}^+)$ is an isometry,  we get
\begin{align*}
\|L_1^*(L_{2, j}^*(L_3^*(\omega_0)))\| \geq \delta_\Psi \varepsilon_1\varepsilon_2 \|\rho_b\|.
\end{align*}
Now, \cref{pro:FrameFlowLNIC,eqn:RingLikeSubset,eqn:ConstantBPPsiAdjoint,eqn:ConstantVarkappa0} also implies that
\begin{align*}
\frac{(L_1^* \circ L_{2, j}^* \circ L_3^*)(\omega_0)}{\|(L_1^* \circ L_{2, j}^* \circ L_3^*)(\omega_0)\|} \in \mathfrak{R}_{\varkappa_0}(\LieSimple^+).
\end{align*}
Thus, by \cref{pro:NonConcentrationProperty}, there exists $\check{x}_2 \in \limitset \cap \left(B_{s_1}^{\mathrm{CC}}(\check{x}_1) \setminus B_{s_2}^{\mathrm{CC}}(\check{x}_1)\right)$ such that
\begin{align*}
|\langle z, (L_1^* \circ L_{2, j}^* \circ L_3^*)(\omega_0) \rangle| &\geq \frac{\epsilon_1}{2\|\rho_b\|} \cdot \varepsilon_3 \cdot \|L_1^*(L_{2, j}^*(L_3^*(\omega_0)))\| \\
&\geq \frac{\epsilon_1}{2\|\rho_b\|} \cdot \varepsilon_1\varepsilon_2\varepsilon_3\delta_\Psi \|\rho_b\| \geq 7\delta_1\epsilon_1
\end{align*}
where $z = \bigl(0, \bigl(n_{\check{x}_1}^+\bigr)^{-1} \cdot \check{x}_2\bigr) \in \T_0(\LieN^+)$.
\end{proof}

Let $(b, \rho) \in \widehat{M}_0(b_0)$, $H \in \mathcal{V}_\rho(\tilde{U})$, $k \in \mathcal{A}$, and $1 \leq r \leq r_k^{(b, \rho)}$. Corresponding to
\begin{align*}
\omega = \frac{\rho_b\big(\Phi^{\alpha_0}\big(v_0\big(x_{k, r, 1}^{(b, \rho)}\big)\big)^{-1}\big)H\big(v_0\big(x_{k, r, 1}^{(b, \rho)}\big)\big)}{\big\|H\big(v_0\big(x_{k, r, 1}^{(b, \rho)}\big)\big)\big\|_2} \in V_\rho^{\oplus \dim(\rho)}
\end{align*}
denote $j_{k, r}^{(b, \rho), H}$ and $x_{k, r, 2}^{(b, \rho), H}$ to be the $j$ and $\Psi(\check{x}_2) \in U_1 \cap \bigl(W_{\epsilon_1/2\|\rho_b\|}^{\mathrm{su}, \mathrm{CC}}\big(x_{k, r, 1}^{(b, \rho)}\big) \setminus W_{\varepsilon_3\epsilon_1/2\|\rho_b\|}^{\mathrm{su}, \mathrm{CC}}\big(x_{k, r, 1}^{(b, \rho)}\big)\bigr)$ provided by \cref{lem:PartnerPointInZariskiDenseLimitSetForBPBound}. Define
\begin{align*}
D_{k, r, 1}^{(b, \rho)} &= W_{\epsilon_2/\|\rho_b\|}^{\mathrm{su}, \mathrm{CC}}\big(x_{k, r, 1}^{(b, \rho)}\big) \subset C_{k, r}^{(b, \rho)}, & D_{k, r, 2}^{(b, \rho), H} &= W_{\epsilon_2/\|\rho_b\|}^{\mathrm{su}, \mathrm{CC}}\big(x_{k, r, 2}^{(b, \rho), H}\big) \subset C_{k, r}^{(b, \rho)}, \\
\slashed{D}_{k, r, 1}^{(b, \rho)} &= W_{\epsilon_2/2\|\rho_b\|}^{\mathrm{su}, \mathrm{CC}}\big(x_{k, r, 1}^{(b, \rho)}\big) \subset C_{k, r}^{(b, \rho)}, & \slashed{D}_{k, r, 2}^{(b, \rho), H} &= W_{\epsilon_2/2\|\rho_b\|}^{\mathrm{su}, \mathrm{CC}}\big(x_{k, r, 2}^{(b, \rho), H}\big) \subset C_{k, r}^{(b, \rho)}, \\
\hat{D}_{k, r, 1}^{(b, \rho)} &= W_{2N\epsilon_2/\|\rho_b\|}^{\mathrm{su}, \mathrm{CC}}\big(x_{k, r, 1}^{(b, \rho)}\big) \subset C_{k, r}^{(b, \rho)}, & \hat{D}_{k, r, 2}^{(b, \rho), H} &= W_{2N\epsilon_2/\|\rho_b\|}^{\mathrm{su}, \mathrm{CC}}\big(x_{k, r, 2}^{(b, \rho), H}\big) \subset C_{k, r}^{(b, \rho)}.
\end{align*}
Denote $\psi_{k, r, 1}^{(b, \rho)}, \psi_{k, r, 2}^{(b, \rho), H} \in C^\infty(\tilde{U}, \mathbb R)$ to be bump functions with $\supp\big(\psi_{k, r, 1}^{(b, \rho)}\big) = \overline{D_{k, r, 1}^{(b, \rho)}}$ and $\supp\big(\psi_{k, r, 2}^{(b, \rho), H}\big) = \overline{D_{k, r, 2}^{(b, \rho), H}}$ such that they attain the maximum values $\psi_{k, r, 1}^{(b, \rho)}\bigr|_{\overline{\slashed{D}_{k, r, 1}^{(b, \rho)}}} = \psi_{k, r, 2}^{(b, \rho), H}\bigr|_{\overline{\slashed{D}_{k, r, 2}^{(b, \rho), H}}} = 1$, and the minimum values $\psi_{k, r, 1}^{(b, \rho)}\bigr|_{\tilde{U} \setminus D_{k, r, 1}^{(b, \rho)}} = \psi_{k, r, 2}^{(b, \rho), H}\bigr|_{\tilde{U} \setminus D_{k, r, 2}^{(b, \rho), H}} = 0$, and we can further assume that $\bigl|\psi_{k, r, 1}^{(b, \rho)}\bigr|_{C^1}, \bigl|\psi_{k, r, 2}^{(b, \rho), H}\bigr|_{C^1} \leq \frac{4\|\rho_b\|}{\epsilon_2}$. It can be checked that $D_{k, r_1, p_1}^{(b, \rho)} \cap D_{k, r_2, p_2}^{(b, \rho), H} = \varnothing$ for all $(r_1, p_1), (r_2, p_2) \in \{1, 2, \dotsc, r_k^{(b, \rho)}\} \times \{1, 2\}$ with $(r_1, p_1) \neq (r_2, p_2)$ and $k \in \mathcal{A}$. Define $\Xi_1(b, \rho) = \big\{(k, r) \in \mathbb Z^2: k \in \mathcal{A}, r \in \big\{1, 2, \dotsc, r_k^{(b, \rho)}\big\}\big\}$ and $\Xi_2 = \{1, 2\} \times \{1, 2\}$. Define $\Xi(b, \rho) = \Xi_1(b, \rho) \times \Xi_2$. For all $(k, r, p, l) \in \Xi(b, \rho)$, denoting $j_{k, r}^{(b, \rho), H}$ by $j$ for convenience, we define the function $\tilde{\psi}_{(k, r, p, l)}^{(b, \rho), H} \in C^\infty(\tilde{U}, \mathbb R)$ by
\begin{align*}
\tilde{\psi}_{(k, r, p, l)}^{(b, \rho), H} =
\begin{cases}
\chi_{\tilde{\mathtt{C}}[\alpha_0]} \cdot \left(\psi_{k, r, 1}^{(b, \rho)} \circ \sigma^{\alpha_0}\right), & p = 1, l = 1 \\
\chi_{\tilde{\mathtt{C}}[\alpha_j]} \cdot \left(\psi_{k, r, 1}^{(b, \rho)} \circ \sigma^{\alpha_j}\right), & p = 1, l = 2 \\
\chi_{\tilde{\mathtt{C}}[\alpha_0]} \cdot \left(\psi_{k, r, 2}^{(b, \rho), H} \circ \sigma^{\alpha_0}\right), & p = 2, l = 1 \\
\chi_{\tilde{\mathtt{C}}[\alpha_j]} \cdot \left(\psi_{k, r, 2}^{(b, \rho), H} \circ \sigma^{\alpha_j}\right), & p = 2, l = 2
\end{cases}
\end{align*}
where using $\sigma^{\alpha_0}$ and $\sigma^{\alpha_j}$ are indeed justified because of the indicator functions $\chi_{\tilde{\mathtt{C}}[\alpha_0]} = \chi_{v_0(\tilde{U}_1)}$ and $\chi_{\tilde{\mathtt{C}}[\alpha_j]} = \chi_{v_j(\tilde{U}_1)}$. For all subsets $J \subset \Xi(b, \rho)$, we define
\begin{align*}
\beta_J^H = \chi_{\tilde{U}} - \mu\sum_{(k, r, p, l) \in J} \tilde{\psi}_{(k, r, p, l)}^{(b, \rho), H} \in C^\infty(\tilde{U}, \mathbb R).
\end{align*}

\begin{remark}
We will often include the superscript $H$ even when there is no dependence on it for a more uniform notation to simplify exposition.
\end{remark}

\begin{lemma}
\label{lem:NumberOfIntersectingBallsLessThanOrEqualToN}
Let $(b, \rho) \in \widehat{M}_0(b_0)$, $H \in \mathcal{V}_\rho(\tilde{U})$, and $J \subset \Xi(b, \rho)$. Then, any connected component of
\begin{align*}
\bigcup \left\{D_{k, r, p}^{(b, \rho), H}: (k, r, p, l) \in J \text{ for some } l \in \{1, 2\}\right\}
\end{align*}
is a union of at most $N$ number of the terms and hence contained in $\hat{D}_{k, r, p}^{(b, \rho), H}$ for any $(k, r, p, l) \in J$ corresponding to one of those terms.
\end{lemma}

\begin{corollary}
\label{cor:SupAndC1SeminormBoundsForBeta_J}
Let $(b, \rho) \in \widehat{M}_0(b_0)$, $H \in \mathcal{V}_\rho(\tilde{U})$, and $J \subset \Xi(b, \rho)$. Then, we have $1 - N\mu \leq \beta_J^H \leq 1$ and $\left|\beta_J^H\right|_{C^1} \leq \frac{4N\mu\|\rho_b\|}{\epsilon_2}$.
\end{corollary}

\begin{definition}[Dolgopyat operator]
For all $\xi \in \mathbb C$ with $|a| < a_0'$, if $(b, \rho) \in \widehat{M}_0(b_0)$, then for all $J \subset \Xi(b, \rho)$ and $H \in \mathcal{V}_\rho(\tilde{U})$, we define the \emph{Dolgopyat operator} $\mathcal{N}_{a, J}^H: C^1(\tilde{U}, \mathbb R) \to C^1(\tilde{U}, \mathbb R)$ by
\begin{align*}
\mathcal{N}_{a, J}^H(h) = \tilde{\mathcal{L}}_{a}^m\big(\beta_J^Hh\big) \qquad \text{for all $h \in C^1(\tilde{U}, \mathbb R)$}.
\end{align*}
\end{definition}

\begin{definition}[Dense]
For all $(b, \rho) \in \widehat{M}_0(b_0)$, a subset $J \subset \Xi(b, \rho)$ is said to be \emph{dense} if for all $(k, r) \in \Xi_1(b, \rho)$, there exists $(p, l) \in \Xi_2$ such that $(k, r, p, l) \in J$.
\end{definition}

For all $(b, \rho) \in \widehat{M}_0(b_0)$, define $\mathcal J(b, \rho) = \{J \subset \Xi(b, \rho): J \text{ is dense}\}$.

\section{Proof of \texorpdfstring{\cref{thm:FrameFlowDolgopyat}}{\autoref{thm:FrameFlowDolgopyat}}}
\label{sec:ProofOfFrameFlowDolgopyat}
The proofs of \cref{itm:FrameFlowDolgopyatProperty1,itm:FrameFlowDolgopyatProperty2,itm:FrameFlowLogLipschitzDolgopyat} in \cref{thm:FrameFlowDolgopyat} are now verbatim repetitions of the proofs in \cite[\S\ 9]{SW21}. Let us say a few words about the proof of \cref{itm:FrameFlowDolgopyatProperty2} in \cref{thm:FrameFlowDolgopyat}. The only sticking point is that we need to be careful to ensure that the Patterson--Sullivan measure has the doubling property with respect to the Carnot--Carath\'{e}odory metric, say on $\LieN^+ = \log(\partial_\infty\mathbb{H}_\mathbb{K}^n \setminus \{e^-\})$. This is indeed true because the Patterson--Sullivan measure has the doubling property with respect to the visual metric which is equivalent to the Carnot--Carath\'{e}odory metric on $\LieN^+ = \log(\partial_\infty\mathbb{H}_\mathbb{K}^n \setminus \{e^-\})$  (see for example \cite[Proposition 3.7]{Bou18} and its proof and \cite[Proposition 3.12]{PPS15}).

Although the proof of \cref{itm:FrameFlowDominatedByDolgopyat} is also similar, we include the main part of the proof because this is where the crucial cancellations occur via \cref{lem:PartnerPointInZariskiDenseLimitSetForBPBound} whose source wasis LNIC, NCP, and the lower bound in \cref{lem:maActionLowerBound}.

\subsection{Proof of \texorpdfstring{\cref{itm:FrameFlowDominatedByDolgopyat}}{Property \ref{itm:FrameFlowDominatedByDolgopyat}} in \texorpdfstring{\cref{thm:FrameFlowDolgopyat}}{\autoref{thm:FrameFlowDolgopyat}}}
Now, for all $\xi \in \mathbb C$ with $|a| < a_0'$, if $(b, \rho) \in \widehat{M}_0(b_0)$, then for all $H \in \mathcal{V}_\rho(\tilde{U})$, $h \in K_{E\|\rho_b\|}(\tilde{U})$, and $1 \leq j \leq j_{\mathrm{m}}$, we define the functions $\chi_{j, 1}^{[\xi, \rho, H, h]}, \chi_{j, 2}^{[\xi, \rho, H, h]}: \tilde{U}_1 \to \mathbb C$ by
\begin{align*}
&\chi_{j, 1}^{[\xi, \rho, H, h]}(u) \\
={}&\frac{\left\|e^{f_{\alpha_0}^{(a)}(v_0(u))} \rho_b(\Phi^{\alpha_0}(v_0(u))^{-1})H(v_0(u)) + e^{f_{\alpha_j}^{(a)}(v_j(u))} \rho_b(\Phi^{\alpha_j}(v_j(u))^{-1})H(v_j(u))\right\|_2}{(1 - N\mu)e^{f_{\alpha_0}^{(a)}(v_0(u))}h(v_0(u)) + e^{f_{\alpha_j}^{(a)}(v_j(u))}h(v_j(u))}
\end{align*}
and
\begin{align*}
&\chi_{j, 2}^{[\xi, \rho, H, h]}(u) \\
={}&\frac{\left\|e^{f_{\alpha_0}^{(a)}(v_0(u))} \rho_b(\Phi^{\alpha_0}(v_0(u))^{-1})H(v_0(u)) + e^{f_{\alpha_j}^{(a)}(v_j(u))} \rho_b(\Phi^{\alpha_j}(v_j(u))^{-1})H(v_j(u))\right\|_2}{e^{f_{\alpha_0}^{(a)}(v_0(u))}h(v_0(u)) + (1 - N\mu)e^{f_{\alpha_j}^{(a)}(v_j(u))}h(v_j(u))}
\end{align*}
for all $u \in \tilde{U}_1$.

\Cref{lem:chiLessThan1} is the main lemma regarding cancellations among the summands of the transfer operator formulated using the above functions. The next two lemmas are required for its proof. The first is proved exactly as in \cite[Lemma 9.8]{SW21} and the second can be proved by elementary trigonometry.

\begin{lemma}
\label{lem:FrameFlowHTrappedByh}
Let $(b, \rho) \in \widehat{M}_0(b_0)$. Suppose that $H \in \mathcal{V}_\rho(\tilde{U})$ and $h \in K_{E\|\rho_b\|}(\tilde{U})$ satisfy \cref{itm:FrameFlowDominatedByh,itm:FrameFlowLogLipschitzh} in \cref{thm:FrameFlowDolgopyat}. Then for all $(k, r, p, l) \in \Xi(b, \rho)$, denoting $0$ by $j$ if $l = 1$ and $j_{k, r}^{(b, \rho), H}$ by $j$ if $l = 2$, we have
\begin{align*}
\frac{1}{2} \leq \frac{h(v_j(u))}{h(v_j(u'))} \leq 2 \qquad \text{for all $u, u' \in \hat{D}_{k, r, p}^{(b, \rho), H}$}
\end{align*}
and also either of the alternatives
\begin{alternative}
\item\label{alt:FrameFlowHLessThan3/4h}	$\|H(v_j(u))\|_2 \leq \frac{3}{4}h(v_j(u))$ for all $u \in \hat{D}_{k, r, p}^{(b, \rho), H}$,
\item\label{alt:FrameFlowHGreaterThan1/4h}	$\|H(v_j(u))\|_2 \geq \frac{1}{4}h(v_j(u))$ for all $u \in \hat{D}_{k, r, p}^{(b, \rho), H}$.
\end{alternative}
\end{lemma}

For any $k \geq 2$, denote by $\Theta(w_1, w_2) = \arccos\left(\frac{\langle w_1, w_2\rangle}{\|w_1\| \cdot \|w_2\|}\right) \in [0, \pi]$ the angle between $w_1, w_2 \in \mathbb R^k \setminus \{0\}$, where we use the standard inner product and norm.

\begin{lemma}
\label{lem:StrongTriangleInequality}
Let $k \geq 2$. If $w_1, w_2 \in \mathbb R^k \setminus \{0\}$ such that $\Theta(w_1, w_2) \geq \alpha$ and $\frac{\|w_1\|}{\|w_2\|} \leq L$ for some $\alpha \in [0, \pi]$ and $L \geq 1$, then we have
\begin{align*}
\|w_1 + w_2\| \leq \left(1 - \frac{\alpha^2}{16L}\right)\|w_1\| + \|w_2\|.
\end{align*}
\end{lemma}

\begin{lemma}
\label{lem:chiLessThan1}
Let $\xi \in \mathbb C$ with $|a| < a_0'$ and $(b, \rho) \in \widehat{M}_0(b_0)$. Suppose $H \in \mathcal{V}_\rho(\tilde{U})$ and $h \in K_{E\|\rho_b\|}(\tilde{U})$ satisfy \cref{itm:FrameFlowDominatedByh,itm:FrameFlowLogLipschitzh} in \cref{thm:FrameFlowDolgopyat}. For all $(k, r) \in \Xi_1(b, \rho)$, denoting $j_{k, r}^{(b, \rho), H}$ by $j$, there exists $(p, l) \in \Xi_2$ such that $\chi_{j, l}^{[\xi, \rho, H, h]}(u) \leq 1$ for all $u \in \hat{D}_{k, r, p}^{(b, \rho), H}$.
\end{lemma}

\begin{proof}
Let $\xi \in \mathbb C$ with $|a| < a_0'$ and $(b, \rho) \in \widehat{M}_0(b_0)$. Suppose $H \in \mathcal{V}_\rho(\tilde{U})$ and $h \in K_{E\|\rho_b\|}(\tilde{U})$ satisfy \cref{itm:FrameFlowDominatedByh,itm:FrameFlowLogLipschitzh} in \cref{thm:FrameFlowDolgopyat}. Let $(k, r) \in \Xi_1(b, \rho)$. Denote $j_{k, r}^{(b, \rho), H}$ by $j$, $x_{k, r, 1}^{(b, \rho)}$ by $x_1$, $x_{k, r, 2}^{(b, \rho), H}$ by $x_2$, and $\hat{D}_{k, r, p}^{(b, \rho), H}$ by $\hat{D}_p$. Now, suppose \cref{alt:FrameFlowHLessThan3/4h} in \cref{lem:FrameFlowHTrappedByh} holds for $(k, r, p, l) \in \Xi(b, \rho)$ for some $(p, l) \in \Xi_2$. Then it is easy to check that $\chi_{j, l}^{[\xi, \rho, H, h]}(u) \leq 1$ for all $u \in \hat{D}_p$. Otherwise, \cref{alt:FrameFlowHGreaterThan1/4h} in \cref{lem:FrameFlowHTrappedByh} holds for $(k, r, 1, 1), (k, r, 1, 2), (k, r, 2, 1), (k, r, 2, 2) \in \Xi(b, \rho)$. We would like to use \cref{lem:StrongTriangleInequality} but first we need to establish bounds on relative angle and relative size. We start with the former. Define $\omega_\ell(u) = \frac{H(v_\ell(u))}{\|H(v_\ell(u))\|_2}$ and $\phi_\ell(u) = \Phi^{\alpha_\ell}(v_\ell(u))$ for all $u \in \tilde{U}_1$ and $\ell \in \{0, j\}$. Let $D = 2\dim(\rho)^2$ and define the map $\varphi: \mathbb R^D \setminus \{0\} \to \mathbb R^D$ by $\varphi(w) = \frac{w}{\|w\|}$ for all $w \in \mathbb R^D \setminus \{0\}$, where we use the standard inner product and norm on $\mathbb R^D$. Note that $\|(d\varphi)_w\|_{\mathrm{op}} = \frac{1}{\|w\|}$ for all $w \in \mathbb R^D$. We can write $\omega_\ell = \varphi \circ H \circ v_\ell$ using the isomorphism $V_\rho^{\oplus \dim(\rho)} \cong \mathbb R^D$ of real vector spaces. Then using \cref{lem:SigmaHyperbolicity,eqn:Constantm2}, we calculate that
\begin{align*}
\|(d\omega_\ell)_u\|_{\mathrm{op}} &\leq \|(d\varphi)_{H(v_\ell(u))}\|_{\mathrm{op}} \|(dH)_{v_\ell(u)}\|_{\mathrm{op}} \|(dv_\ell)_u\|_{\mathrm{op}} \\
&\leq \frac{1}{\|H(v_\ell(u))\|_2} \cdot E\|\rho_b\|h(v_\ell(u)) \cdot \frac{1}{c_0\kappa_2^{m_2}} \\
&\leq \frac{4E\|\rho_b\|}{c_0\kappa_2^{m_2}} \leq \delta_1\|\rho_b\|
\end{align*}
for all $u \in \hat{D}_p$, $\ell \in \{0, j\}$, and $p \in \{1, 2\}$. In other words, $\omega_0$ and $\omega_j$ are Lipschitz on $\hat{D}_p$ with Lipschitz constant $\delta_1\|\rho_b\|$ for all $p \in \{1, 2\}$. Define
\begin{align*}
\begin{split}
V_\ell(u) &= e^{f_{\alpha_\ell}^{(a)}(v_\ell(u))} \rho_b(\phi_\ell(u)^{-1})H(v_\ell(u)), \\
\hat{V}_\ell(u) &= \frac{V_\ell(u)}{\|V_\ell(u)\|_2} = \rho_b(\phi_\ell(u)^{-1})\omega_\ell(u),
\end{split}
\qquad
\text{for all $u \in \tilde{U}_1$ and $\ell \in \{0, j\}$}.
\end{align*}
Since $\omega_0$ and $\omega_j$ are Lipschitz and $d(x_1, x_2) \leq \frac{\epsilon_1}{2\|\rho_b\|}$, we have
\begin{align*}
&\big\|\hat{V}_0(x_2) - \hat{V}_j(x_2)\big\|_2 \\
={}&\|\rho_b(\phi_0(x_2)^{-1})\omega_0(x_2) - \rho_b(\phi_j(x_2)^{-1})\omega_j(x_2)\|_2 \\
={}&\|\rho_b(\phi_j(x_2)\phi_0(x_2)^{-1})\omega_0(x_2) - \omega_j(x_2)\|_2 \\
\geq{}&\|\rho_b(\phi_j(x_2)\phi_0(x_2)^{-1})\omega_0(x_1) - \omega_j(x_1)\|_2 \\
{}&- \|\rho_b(\phi_j(x_2)\phi_0(x_2)^{-1})\omega_0(x_2) - \rho_b(\phi_j(x_2)\phi_0(x_2)^{-1})\omega_0(x_1)\|_2 \\
{}&- \|\omega_j(x_2) - \omega_j(x_1)\|_2 \\
={}&\|\rho_b(\phi_j(x_2)\phi_0(x_2)^{-1})\omega_0(x_1) - \omega_j(x_1)\|_2 - \|\omega_0(x_2) - \omega_0(x_1)\|_2 \\
{}&- \|\omega_j(x_2) - \omega_j(x_1)\|_2 \\
\geq{}&\|\rho_b(\phi_j(x_2)\phi_0(x_2)^{-1})\omega_0(x_1) - \rho_b(\phi_j(x_1)\phi_0(x_1)^{-1})\omega_0(x_1)\|_2 \\
{}&- \|\rho_b(\phi_j(x_1)\phi_0(x_1)^{-1})\omega_0(x_1) - \omega_j(x_1)\|_2 - \delta_1\epsilon_1 \\
={}&\|\rho_b(\phi_0(x_1)^{-1})\omega_0(x_1) - \rho_b(\phi_0(x_1)^{-1}\phi_0(x_2)\phi_j(x_2)^{-1}\phi_j(x_1)\phi_0(x_1)^{-1})\omega_0(x_1)\|_2 \\
{}&- \|\rho_b(\phi_0(x_1)^{-1})\omega_0(x_1) - \rho_b(\phi_j(x_1)^{-1})\omega_j(x_1)\|_2 - \delta_1\epsilon_1 \\
\geq{}&\|\rho_b(\phi_0(x_1)^{-1})\omega_0(x_1) - \rho_b(\BP_j(x_1, x_2))\rho_b(\phi_0(x_1)^{-1})\omega_0(x_1)\|_2 \\
&{}- \big\|\hat{V}_0(x_1) - \hat{V}_j(x_1)\big\|_2 - \delta_1\epsilon_1.
\end{align*}
Denote $\omega = \rho_b(\phi_0(x_1)^{-1})\omega_0(x_1)$ and $Z = d(\BP_{j, x_1} \circ \iota(n_{\check{x}_1}^+) \circ \Psi)_0(z)$ where we take $z = \bigl(0, \bigl(n_{\check{x}_1}^+\bigr)^{-1} \cdot \check{x}_2\bigr) \in \T_0(\LieN^+)$. Recalling the definition of $\Psi$ and that $\iota(n_{\check{x}_1}^+)$ is an isometry, we can conclude that $z' = d(\iota(n_{\check{x}_1}^+) \circ \Psi)_0(z) \in \T_{x_1}(\mathcal{U})$ such that the geodesic $\gamma_{z'}$ has end points $\gamma_{z'}(0) = x_1$ and $\gamma_{z'}(1) = x_2$. Hence, continuing to bound the first term above, we can apply \cref{lem:PartnerPointInZariskiDenseLimitSetForBPBound,lem:ComparingExpWithBP,eqn:Constantepsilon1} to get
\begin{align*}
&\|\omega - \rho_b(\BP_j(x_1, x_2))(\omega)\|_2 \\
\geq{}&\|\omega - \rho_b(\exp(Z))(\omega)\|_2 - \|\rho_b(\exp(Z))(\omega) - \rho_b(\BP_j(x_1, x_2))(\omega)\|_2 \\
\geq{}&\|\omega - \exp(d\rho_b(Z))(\omega)\|_2 - \|\rho_b\| \cdot d_{AM}(\exp(Z), \BP_j(x_1, x_2)) \\
\geq{}&\|d\rho_b(Z)(\omega)\|_2 - \|\rho_b\|^2 \|Z\|^2 - \|\rho_b\| \cdot d_{AM}(\exp(Z), \BP_j(x_1, x_2)) \\
\geq{}&\|d\rho_b(Z)(\omega)\|_2 - \|\rho_b\|^2 C_{\BP}^2 d(x_1, x_2)^2 - C_{\exp, \BP} \cdot \|\rho_b\| \cdot d(x_1, x_2)^2 \\
\geq{}&7\delta_1\epsilon_1 - \delta_1\epsilon_1 - \delta_1\epsilon_1 \geq 5\delta_1\epsilon_1.
\end{align*}
Hence, we have
\begin{align*}
\big\|\hat{V}_0(x_1) - \hat{V}_j(x_1)\big\|_2 + \big\|\hat{V}_0(x_2) - \hat{V}_j(x_2)\big\|_2 \geq 4\delta_1\epsilon_1.
\end{align*}
Then, $\big\|\hat{V}_0(x_p) - \hat{V}_j(x_p)\big\|_2 \geq 2\delta_1\epsilon_1$ for some $p \in \{1, 2\}$. Recalling estimates from \cref{lem:FrameFlowPreliminaryLogLipschitz}, \cref{eqn:Constantepsilon2}, that $\rho_b$ is a unitary representation, and that $\omega_\ell$ is Lipschitz, we have
\begin{align*}
&\big\|\hat{V}_\ell(x_p) - \hat{V}_\ell(u)\big\|_2 \\
\leq{}&\|(\rho_b(\phi_\ell(x_p)^{-1}) - \rho_b(\phi_\ell(u)^{-1}))\omega_\ell(x_p)\|_2 + \|\rho_b(\phi_\ell(u)^{-1})(\omega_\ell(x_p) - \omega_\ell(u))\|_2 \\
\leq{}&A_0 \|\rho_b\| d(x_p, u) + \delta_1\|\rho_b\| d(x_p, u) \\
\leq{}&(A_0 + \delta_1)\|\rho_b\| \cdot \frac{2N\epsilon_2}{\|\rho_b\|} = 2N\epsilon_2(A_0 + \delta_1) \leq \frac{\delta_1\epsilon_1}{2}
\end{align*}
for all $u \in \hat{D}_p$ and $\ell \in \{0, j\}$. Hence $\big\|\hat{V}_0(u) - \hat{V}_j(u)\big\|_2 \geq \delta_1\epsilon_1 \in (0, 1)$ for all $u \in \hat{D}_p$. Then using the cosine law, the required bound for relative angle is
\begin{align*}
\Theta(V_0(u), V_j(u)) = \Theta(\hat{V}_0(u), \hat{V}_j(u)) \geq \arccos\left(1 - \frac{(\delta_1\epsilon_1)^2}{2}\right) \in (0, \pi).
\end{align*}
For the bound on relative size, let $(\ell, \ell') \in \{(0, j), (j, 0)\}$ such that $h(v_\ell(u_0)) \leq h(v_{\ell'}(u_0))$ for some $u_0 \in \hat{D}_p$. Let $l = 1$ if $(\ell, \ell') = (0, j)$ and $l = 2$ if $(\ell, \ell') = (j, 0)$. Recalling that $\rho_b$ is a unitary representation, by \cref{lem:FrameFlowHTrappedByh}, we have
\begin{align*}
\frac{\|V_\ell(u)\|_2}{\|V_{\ell'}(u)\|_2} &= \frac{e^{f_{\alpha_\ell}^{(a)}(v_\ell(u))}\|H(v_\ell(u))\|_2}{e^{f_{\alpha_{\ell'}}^{(a)}(v_{\ell'}(u))}\|H(v_{\ell'}(u))\|_2} \leq \frac{4e^{f_{\alpha_\ell}^{(a)}(v_\ell(u)) - f_{\alpha_{\ell'}}^{(a)}(v_{\ell'}(u))}h(v_{\ell}(u))}{h(v_{\ell'}(u))} \\
&\leq \frac{16e^{2m_2 T_0}h(v_{\ell}(u_0))}{h(v_{\ell'}(u_0))} \leq 16e^{2m_2 T_0}
\end{align*}
for all $u \in \hat{D}_p$, which is the required bound on relative size. Using \cref{lem:StrongTriangleInequality}, \cref{eqn:Constantmu}, and $\|H\| \leq h$ for $\|V_\ell(u) + V_{\ell'}(u)\|_2$ gives $\chi_{j, l}^{[\xi, \rho, H, h]}(u) \leq 1$ for all $u \in \hat{D}_p$.
\end{proof}

With \cref{lem:chiLessThan1} in hand, it is straightforward to derive the following lemma whose proof is exactly as in \cite[Lemma 9.11]{SW21}.

\begin{lemma}
For all $\xi \in \mathbb C$ with $|a| < a_0'$, if $(b, \rho) \in \widehat{M}_0(b_0)$, and if $H \in \mathcal{V}_\rho(\tilde{U})$ and $h \in K_{E\|\rho_b\|}(\tilde{U})$ satisfy \cref{itm:FrameFlowDominatedByh,itm:FrameFlowLogLipschitzh} in \cref{thm:FrameFlowDolgopyat}, then there exists $J \in \mathcal{J}(b, \rho)$ such that
\begin{align*}
\big\|\tilde{\mathcal{M}}_{\xi, \rho}^m(H)(u)\big\|_2 \leq \mathcal{N}_{a, J}^H(h)(u) \qquad \text{for all $u \in \tilde{U}$}.
\end{align*}
\end{lemma}

\nocite{*}
\bibliographystyle{alpha_name-year-title}
\bibliography{References}
\end{document}